\documentclass[11pt,a4paper]{amsart}%
\usepackage{amsfonts,amsmath,amssymb}
\usepackage{hyperref}
\usepackage{amssymb,amsthm,amsxtra}
\usepackage[usenames]{color}
\usepackage{amscd}
\usepackage{amsthm}
\usepackage{amsfonts}
\usepackage{amssymb}
\usepackage{amsmath}
\usepackage{graphicx}
\usepackage{enumerate}%
\usepackage[all]{xy}
\usepackage{aurical}
\usepackage[T1]{fontenc}
\let\oldtocsubsection=\tocsubsection
\let\oldtocsubsubsection=\tocsubsubsection
\renewcommand{\tocsubsection}[2]{\hspace{1em}\oldtocsubsection{#1}{#2}}
\renewcommand{\tocsubsubsection}[2]{\hspace{2em}\oldtocsubsubsection{#1}{#2}}

\theoremstyle{plain}

\newtheorem*{theorem*}{Theorem}
\newtheorem{lemma}{Lemma}[subsection]

\newtheorem{theorem}[lemma]{Theorem}
\newtheorem{corollary}[lemma]{Corollary}

\newtheorem*{conjecture*}{Conjecture}

\newtheorem{thm}[lemma]{Theorem}
\newtheorem{prop}[lemma]{Proposition}
\newtheorem{lem}[lemma]{Lemma}
\newtheorem{cor}[lemma]{Corollary}

\newtheorem{introtheorem}{Theorem}

\newtheorem{introcor}[introtheorem]{Corollary}
\newtheorem{introthm}[introtheorem]{Theorem}

\theoremstyle{definition}

\newtheorem{definition}[lemma]{Definition}
\newtheorem{defn}[lemma]{Definition}

\newtheorem{notn}[lemma]{Notation}

\newtheorem{example}[lemma]{Example}

\newtheorem*{example*}{Example}

\theoremstyle{remark}

\newtheorem*{remark*}{Remark}

\newtheorem{remark}[lemma]{Remark}

 \theoremstyle{plain}

\newcommand{\ind}{\operatorname{ind}}

\newcommand{\ad}{\operatorname{ad}}

\newcommand{\Rep}{\operatorname{Rep}}
\newcommand{\Hom}{\operatorname{Hom}}

\newcommand{\Rad}{\operatorname{Rad}}
\newcommand{\oH}{\operatorname{H}}

\newcommand{\diag}{\operatorname{diag}}

\newcommand{\A}{\mathbb{A}}

\newcommand{\eps}{\varepsilon}

\newcommand{\length}{\operatorname{length}}

\renewcommand{\Im}{\operatorname{Im}}
\newcommand{\Ker}{\operatorname{Ker}}

\newcommand{\Z}{{\mathbb Z}}
\newcommand{\Q}{{\mathbb Q}}
\newcommand{\R}{{\mathbb R}}

\newcommand{\C}{{\mathbb C}}

\newcommand{\proofend}{\hfill$\Box$\smallskip}

\newcommand{\Exp}{\operatorname{Exp}}

\newcommand{\Id}{\operatorname{Id}}

\newcommand{\alp}{{\alpha}}

\newcommand{\gam}{{\gamma}}

\newcommand{\lam}{{\lambda}}

\newcommand{\Fre}{{Fr\'{e}chet \,}}

\newcommand{\bfG}{{\mathbf{G}}}
\newcommand{\bfL}{{\mathbf{L}}}

\newcommand{\bfN}{{\mathbf{N}}}
\newcommand{\bfU}{{\mathbf{U}}}

\newcommand{\cE}{{\mathcal{E}}}
\newcommand{\cF}{{\mathcal{F}}}

\newcommand{\cH}{{\mathcal{H}}}

\newcommand{\g}{{\mathfrak{g}}}

\newcommand{\fa}{{\mathfrak{a}}}
\newcommand{\fb}{{\mathfrak{b}}}

\newcommand{\fg}{{\mathfrak{g}}}

\newcommand{\fn}{{\mathfrak{n}}}
\newcommand{\fk}{{\mathfrak{k}}}
\newcommand{\fu}{{\mathfrak{u}}}
\newcommand{\fv}{{\mathfrak{v}}}
\newcommand{\fw}{{\mathfrak{w}}}
\newcommand{\fz}{{\mathfrak{z}}}

\newcommand{\cO}{{\mathcal{O}}}
\newcommand{\GL}{\operatorname{GL}}

\newcommand{\SL}{\operatorname{SL}}
\newcommand{\Sp}{\operatorname{Sp}}

\newcommand{\SO}{\operatorname{SO}}
\newcommand{\gl}{{\mathfrak{gl}}}

\newcommand{\Sc}{{\mathcal S}}
\newcommand{\Lie}{\operatorname{Lie}}

\newcommand{\fp}{\mathfrak{p}}
\newcommand{\fl}{\mathfrak{l}}
\newcommand{\fm}{\mathfrak{m}}

\newcommand{\fr}{\mathfrak{r}}
\newcommand{\fs}{\mathfrak{s}}

\newcommand{\cW}{\mathcal{W}}

\newcommand{\Dima}[1]{{{#1}}}
\newcommand{\DimaA}[1]{{{#1}}}
\newcommand{\DimaB}[1]{{{#1}}}
\newcommand{\DimaC}[1]{{{#1}}}
\newcommand{\DimaD}[1]{{{#1}}}
\newcommand{\DimaE}[1]{{{#1}}}
\newcommand{\DimaF}[1]{{{#1}}}
\newcommand{\DimaG}[1]{{{#1}}}
\newcommand{\DimaH}[1]{{{#1}}}
\newcommand{\DimaK}[1]{{{#1}}}
\newcommand{\DimaL}[1]{{{#1}}}
\newcommand{\DimaM}[1]{{{#1}}}
\newcommand{\DimaN}[1]{{{#1}}}
\newcommand{\DimaO}[1]{{{#1}}}
\newcommand{\DimaP}[1]{{{#1}}}
\newcommand{\DimaQ}[1]{{{#1}}}
\newcommand{\DimaR}[1]{{{#1}}}
\newcommand{\DimaS}[1]{{{#1}}}
\newcommand{\DimaT}[1]{{{#1}}}
\newcommand{\DimaU}[1]{{{#1}}}

\newcommand{\onto}{{\twoheadrightarrow}}

\newcommand{\into}{{\hookrightarrow}}

\newcommand{\Orb}{\mathcal{O}}
\newcommand{\WF}{\operatorname{WF}}
\newcommand{\WO}{\operatorname{WO}}
\newcommand{\WS}{\operatorname{WS}}
\newcommand{\QWS}{\operatorname{QWS}}
\newcommand{\QWO}{\operatorname{QWO}}

\renewcommand{\sl}{\mathfrak{sl}}
\newcommand{\F}{{F}}

\begin{document}

\author{Raul Gomez}
\address{Raul Gomez, UANL FCFM Av. Universidad, San Nicolas de los Garza, N.L., Mexico}
\email{raul.gomez.rgm@gmail.com}

\author{Dmitry Gourevitch}
\address{Dmitry Gourevitch,
Faculty of Mathematics and Computer Science,
Weizmann Institute of Science,
234 Herzl Street, Rehovot 7610001 Israel}
\email{dmitry.gourevitch@weizmann.ac.il}
\urladdr{\url{http://www.wisdom.weizmann.ac.il/~dimagur}}

\author{Siddhartha Sahi}
\address{Siddhartha Sahi, Department of Mathematics, Rutgers University, Hill Center -
Busch Campus, 110 Frelinghuysen Road Piscataway, NJ 08854-8019, USA}
\email{sahi@math.rugers.edu}

\date{\today}
\title{Whittaker supports for representations of reductive groups}

\keywords{Fourier coefficient,    wave-front set, oscillator representation, Heisenberg group, metaplectic group, admissible orbit, distinguished orbit, cuspidal representation, automorphic form.
2010 MS Classification:         20G05, 20G20, 20G25, 20G30, 20G35, 22E27, 22E46, 22E50, 22E55, 17B08.
}

\begin{abstract}
Let $F$ be either $\R$ or
a finite extension of $\Q_p$, and let $G$ be \DimaE{a finite central extension of} the group of $F$-points of a reductive group defined over  $F$. Also let  $\pi$ be  a smooth representation of $G$ (\Fre of moderate growth if $F=\R$). For each nilpotent orbit $\cO$ we consider a certain Whittaker quotient $\pi_{\cO}$ of $\pi$.
We define the Whittaker support $\WS(\pi)$ to be the set of maximal $\cO$ among those for which $\pi_{\cO}\neq 0$.

In this paper we prove that all $\cO\in\WS(\pi)$ are quasi-admissible nilpotent orbits, generalizing some of the results in \cite{Mog,JLS}. If $F$ is $p$-adic and $\pi$ is quasi-cuspidal then we show that all $\cO\in\WS(\pi)$  are $F$-distinguished, i.e. do not intersect the Lie algebra of any proper Levi subgroup of $G$ defined over $F$.

We also give an adaptation of our argument to automorphic representations, generalizing some results from \DimaA{\cite{GRS_Sp,Shen,JLS,Cai}} and confirming \DimaL{some conjectures of} \cite{Ginz}.

\DimaA{Our methods are a synergy of the methods of the above-mentioned papers, and of our preceding paper \cite{GGS}.}

\end{abstract}

\maketitle
\tableofcontents


%

\section{Introduction}

The study of Whittaker and degenerate Whittaker models for
representations of reductive groups over local fields evolved in
connection with the theory of automorphic forms (via their Fourier
coefficients), and has found important applications in both areas.
See for example \cite{Sh74,NPS73,Kos,Ka85,Ya86, WaJI,Ginz,Ji07,GRS_Book,Ginz2}.

Let $F$ be  either $\R$ or a finite extension of $\Q_p$, and let $G$ be \DimaE{a finite central extension of} the group of $F$-points of a connected reductive algebraic group defined over $F$. Let $\Rep^{\infty}(G)$ denote the category of smooth representations of $G$ (see \S \ref{subsec:ScInd} below).
Let $\fg$ denote the Lie algebra of $G$ and $\fg^*$ denote its dual space. \DimaG{To every coadjoint nilpotent orbit $\cO\subset \fg^*$ and every $\pi\in\Rep^{\infty}(G)$ we associate a certain generalized Whittaker quotient $\pi_{\cO}$ (see \S \ref{subsec:DefMod} below)}. Let $\WO(\pi)$ denote the set of all nilpotent orbits $\cO$ with $\pi_{\cO}\neq 0$ and $\WS(\pi)$ the set of maximal orbits in $\WO(\pi)$  with respect to the closure ordering.

We recall the notion of admissible nilpotent orbit. It has to do with splitting of a certain metapletic double cover of  the centralizer $G_{\varphi}$ for any $\varphi$ in the orbit.
\DimaA{We also define a \DimaB{  weaker} notion of a quasi-admissible orbit.}
\DimaC{See \S \ref{subsec:cov} below for both notions.}

For split $p$-adic groups, admissibility is also related to the notion of a special nilpotent orbit in the sense of Lusztig (see \S\ref{subsubsec:RelSpecOrb} below). \DimaB{In particular, for $p$-adic classical groups the two notions are equivalent (\cite{Nev}).}


\begin{introthm}[\S \ref{sec:PfAdm}]\label{thm:adm}
Let $\pi\in \Rep^{\infty}(G)$ and let $\cO\in \WS(\pi)$. Then
$\cO$ is a \DimaA{quasi-}admissible orbit.
\end{introthm}
\DimaB{Note that in the Archimedean case, $\cO$ is not always admissible, \emph{e.g.} for minimal representations of $U(2,1)$, see \S\ref{subsubsec:RelSpecOrb} below. The notions of admissible and quasi-admissible also differ for the split real forms of $E_7$ and $E_8$, though we do not know whether the non-admissible quasi-admissible orbits appear in Whittaker supports of representations.
For the symplectic and orthogonal groups the two notions are equivalent.
\begin{prop}[\S \ref{subsubsec:RelSpecOrb}]\label{prop:ClasQuas}
Let $G$ be \DimaC{either $\Sp_{2n}(F)$, or $O(V)$ or $\SO(V)$ (for a quadratic space $V$ over $F$),}  and let $\cO\subset \fg^*$ be a nilpotent orbit. Then the following are equivalent.
\begin{enumerate}[(a)]
\item $\cO$ is admissible
\item $\cO$ is quasi-admissible
\item $\cO$ is special
\end{enumerate}
\end{prop}
We deduce this proposition from \cite{Oht,Nev}.
It is possible  that the notions of admissible and quasi-admissible are equivalent for all $G$ in the case when $F$ is non-Archimedean.}
In this  case, and under the additional assumption that $G$ is \DimaB{classical}, it is shown in \cite{Mog,JLS} that all $\cO\in \WS(\pi)$ are admissible, for any $\pi\in \Rep^{\infty}(G)$.
For exceptional $G$, a slightly weaker result is shown in \cite{JLS}.

For p-adic $F$ it is conjectured that \DimaE{if $G$ is algebraic then} all  orbits in $\WS(\pi)$ are special.
\DimaB{For classical $G$ this follows from \cite{Mog,JLS}.} For $G=G_2(F)$ this follows from \cite{JLS,LS}.
For $F=\R$ the analogous statement does not hold. Namely,    \cite{VogG2} constructs a small unitary irreducible representation $\pi$ of $G_2(\R)$. We show in \S \ref{subsubsec:RelSpecOrb} that $\WS(\pi)$ is also small, i.e. consists of the minimal orbit for $G_2$. This orbit is non-special but admissible.
\DimaD{For classical \DimaE{algebraic} groups over all local fields, all special orbits are quasi-admissible. It is possible that this holds for all groups.}

%

It is quite probable that  if \DimaE{$G$ is algebraic and} $\pi$ is admissible and has integral infinitesimal character then all  $\cO\in \WS(\pi)$ are special, cf. \cite[Theorem D]{BVClass} and \cite[Theorem 1.1]{BVExc}.

We also prove that for quasi-cuspidal $\pi\in \Rep^{\infty}(G)$, the  orbits in $\WS(\pi)$ are $F$-distinguished. Here, $F$ is non-Archimedean, and $\pi$  is  quasi-cuspidal if all its Jacquet reductions vanish (i.e. $r_P\pi=0$ for any parabolic subgroup $P\subset G$) and a nilpotent orbit $\cO\subset \fg^*$ is $F$-distinguished if the corresponding orbit in $\fg$ does not intersect the Lie algebra of a Levi subgroup of any proper parabolic subgroup  $P\subset G$ defined over $F$.
\begin{introtheorem}[\S \ref{subsec:Cusp}]\label{thm:cusp}
Let $F$ be non-Archimedean and $\pi\in \Rep^{\infty}(G)$ be a quasi-cuspidal representation.   Then all
$\cO\in \WS(\pi)$ are $F$-distinguished.
\end{introtheorem}

For classical $G$,  it was shown  in \cite{Mog} that  all the  orbits in the Whittaker support of all tempered admissible (finitely generated) $\pi$ are $F$-distinguished.
\DimaH{For a similar result in} the case $F=\R$  see \cite{Harris}.

\DimaP{One of our basic tools is Lemma \ref{lem:Heis} below, that follows from the Stone-von-Neumann theorem. An analogous lemma first appeared in the non-Archimedean case in   \cite[Lemma 2.2]{GRS}, and is sometimes referred to as  the ``root exchange" lemma.} 
\DimaO{We also use \cite[Lemma 5.10]{BZ} (see Lemma \ref{lem:BZ} below)  and its Archimedean analog Proposition \ref{prop:AV0} that we prove in \S \ref{sec:Arch} below using some properties of modules over  algebras of Schwartz functions established in \cite{dCl}. Two more central tools we use are the deformation technique of \cite{GGS} and the notion of quasi-Whittaker quotients introduced in this paper, see \S \ref{subsec:basic} below. For our strategy of proof see \ref{subsec:struc} below. }

\subsection{Additional results}
 Let $\gamma=(f,h,e)$ be an $\sl_2$-triple.
Let $G_{\gamma}$ be the centralizer of $\gamma$ in $G$, and $\tilde G_{\gamma}$ be its metaplectic cover (see \S \ref{subsec:cov} below). Let $\varphi\in \fg^*$  be given by the Killing form pairing with $f$. Then $\tilde G_{\gamma}$ acts on the generalized Whittaker quotient $\pi_{\varphi}=\pi_{\gamma}$ (see \S \ref{subsec:cov} below).

We denote by $M_{\gamma}$  the  subgroup of $G_{\gamma}$ \DimaA{ generated by the unipotent elements.}
Let $\widetilde{M_{\gamma}}$ denote the preimage of  $M_{\gamma}$ under the projection $\widetilde{G_{\gamma}}\to G_{\gamma}$.
\begin{introthm}[\S \ref{subsec:MaxFin}]\label{thm:MaxFin}
Let $\pi\in \Rep^{\infty}(G)$ and assume that $G \cdot \varphi \in \WS(\pi)$.
Then
\DimaP{
\begin{enumerate}[(i)]
\item If $F$ is non-Archimedean then $\widetilde{M_{\gamma}}$ acts on $\pi_{\varphi}$ by $\pm \Id$.

\item If $F$ is Archimedean then
 the action of $\widetilde{M_{\gamma}}$ on the dual space $\pi_{\varphi}^*$ is locally finite.
\end{enumerate}
}
\end{introthm}

\DimaH{
Let $S\in \fg$ be such that the adjoint action $ad(S)$ diagonalizes over $\Q$ and $ad(S)^*(\varphi)=-2\varphi$.
\DimaH{We will call such pairs $(S,\varphi)$ \emph{Whittaker pairs}}.
Following \cite{MW} we attach to $(S,\varphi)$ a certain  degenerate Whittaker quotient $\pi_{S,\varphi}$. If  $S=h$ then this is the  generalized (a.k.a. neutral) Whittaker quotient (see \S \ref{subsec:DefMod} for the definitions).
\begin{introthm}[\S \ref{sec:MaxOrb}]\label{thm:MaxModAction}
Let $\pi\in \Rep^{\infty}(G)$ and let $(S,\varphi)$ be a Whittaker pair such that $G \cdot \varphi \in \WS(\pi)$.  Then $\pi_{S,\varphi}\neq 0$. Moreover, if $F$ is non-Archimedean then the epimorphism $\pi_{\varphi}\onto \pi_{S,\varphi}$ constructed in \cite[Theorem A]{GGS} is an isomorphism.
\end{introthm}



The special case of $p$-adic $F$ and admissible $\pi$ follows from \cite{MW,Var,GGS}.

Another natural question to ask is - given $\WS(\pi)$, what smaller orbits lie in $\WO(\pi)$? In the case of $\GL_n$ the answer is - all the orbits lying in the closure of orbits in $\WS(\pi)$. For general reductive groups we deduce from Theorem \ref{thm:MaxModAction} a partial result, see Theorem \ref{thm:ComparOrbit} below.

\subsection{Global case}

\DimaD{We also \DimaB{provide}  global analogs  of Theorems \ref{thm:adm}-\ref{thm:MaxModAction}, see \S \ref{sec:Glob} below.
These analogs generalize several results from \cite{GRS_Sp,Shen,Cai,JLS}. \DimaL{In particular, this puts restrictions on Whittaker supports of cuspidal representations, confirming conjectures from \cite[\S 4]{Ginz}.}} We also provide an analog of Theorem \ref{thm:ComparOrbit} and deduce the following corollary.
\begin{introcor}\label{intcor:GlobGL}
Let $K$ be a number field and let $\pi$ be an automorphic representation of $\mathrm{GL}_n(K)$. Let  $\cO\in \WS(\pi)$ and let $\cO'\subset \overline{\cO}$, where $\overline{\cO}$ denotes the Zariski closure of $\cO$ in $\gl_n(K)$. Then $\cO'\in \WO(\pi)$.
\end{introcor}
\DimaK{For a certain analogous statement for $SL_n(K)$ see Corollary \ref{cor:SLGlob} below.} \DimaN{For other groups, we only have a partial result - see Theorem \ref{thm:GlobComparOrbit} below.}
}

\DimaO{
\subsection{Open questions}
Let us summarize some open questions that arise naturally from the results discussed above.
\begin{enumerate}
\item Over a non-archimedean $F$ - do there exist representations of linear reductive groups with non-special Whittaker support?

\item Over a non-archimedean $F$ - are all special orbits admissible?

\item Analogs of the two questions above over a global field $K$.

\item Over $\R$, do non-admissible quasi-admissible orbits of split groups appear in Whittaker supports of representations?

\item Over all fields - are all special orbits quasi-admissible?

\item Given $\WS(\pi)$, how does $\WO(\pi)$ look like?
\end{enumerate}
Another deep conjecture, posed in \cite{MW}, says that for any irreducible $\pi$, all the orbits in $\WS(\pi)$ lie in the same orbit over the algebraic closure. This is conjectured for all global and local fields, but known only in some examples, in particular for $GL_n$, see \cite[Ch. II]{MW} for the non-Archimedean case and \cite[Theorem D]{Ros} and \cite[Theorem B]{GGS} for the Archimedean case. For further open questions we refer the reader to \cite[\S 5]{Ginz}, and \cite[\S 1]{GS_survey}.}
%
\subsection{Structure of the paper}\label{subsec:struc}
In \S \ref{sec:prel} we give the necessary preliminaries on $\sl_2$-triples, smooth representations, oscillator representations, Schwartz induction, generalized and degenerate Whittaker models and covering groups.

In \S \ref{sec:Arch} we prove several statements on non-generic $\pi\in \Rep^{\infty}(B)$.
Here, $B$\ is a Borel subgroup of the metaplectic group $\widetilde{SL_2(\R)}$, and we say that $\pi$ is non-generic if it has no non-zero functionals equivariant under the nilradical $N$ of $B$ by a non-trivial unitary character. Over $p$-adic fields, \cite[Lemma 5.10]{BZ} implies that the action of $N$ on $\pi$ is trivial. Over $\R$ we prove in Proposition \ref{prop:AV0} that for any non-generic $\pi$, the action of the Lie algebra $\fn$ of $N$ on $\pi^*$ is locally nilpotent.

In \S \ref{sec:MaxOrb} we   prove Theorem \ref{thm:MaxModAction}. The proof uses the epimorphism $\cW_{\varphi}\onto \cW_{S,\varphi}$ constructed in \cite{GGS}. \Dima{Here, $\cW_{\varphi}$ and $\cW_{S,\varphi}$ are degenerate Whittaker models, that define the quotients $\pi_{\varphi}$ and $\pi_{S,\varphi}$ as the coivariants $\pi_{\varphi}=(\cW_{\varphi}\otimes \pi)_G$ and $\pi_{S,\varphi}=(\cW_{S,\varphi}\otimes \pi)_G$ (see \S \ref{subsec:DefMod} below). Let us recall the construction of the epimorphism $\cW_{\varphi}\onto \cW_{S,\varphi}$}.
\DimaH{One can show that $S$ can be presented as $h+Z$, where $h$ is a neutral element for $\varphi$ and $Z$ commutes with $h$ and $\varphi$.}
Consider a deformation $S_t=h+tZ$, and denote by $\fu_t$ the sum of eigenspaces of $ad(S_t)$ with eigenvalues at least 1. We call a rational number $0< t <1 $ {\it regular} if $\fu_t=\fu_{t+\eps}$ for any small enough rational $\eps$, and {\it critical} otherwise. Note that there are finitely many critical numbers, and denote them by $t_1<\dots<t_n$. Denote also $t_0:=0$ and $t_{n+1}:=1$. For each $t$ we define two subalgebras $\fl_t,\fr_t\subset \fu_t$. Both $\fl_t$ and $\fr_t$ are maximal isotropic subspaces with respect to the form $\omega_{\varphi}$, $\fr_t$ contains all the eigenspaces of $Z$ in $\fu_t$ with positive eigenvalues and $\fl_t$ contains all the eigenspaces with negative eigenvalues. Note that the restrictions of $\varphi$ to $\fl_t$ and $\fr_t$ define characters of these subalgebras.
Let $L_t:=\Exp(\fl_t)$ and $R_t:=\Exp(\fr_t)$ denote the corresponding subgroups and $\chi_\varphi$ denote their characters defined by $\varphi$.
The Stone-von-Neumann theorem implies
\begin{equation}\label{=RE}
\cW_{S_t,\varphi}\simeq \ind_{L_t}^G(\chi_{\varphi})\simeq \ind_{R_t}^G(\chi_{\varphi}).
\end{equation}
\DimaP{This is an analog of the root exchange lemma \cite[Lemma 2.2]{GRS}.}
We show that for any $0\leq i\leq n, \, \fr_{t_i}\subset \fl_{t_{i+1}}$. This gives a natural epimorphism
$$\cW_{S_{t_i},\varphi}\simeq \ind_{\DimaS{R_{t_{i}}}}^G(\chi_{\varphi})\onto \ind_{\DimaS{L_{t_{i+1}}}}^G(\chi_{\varphi})\simeq \cW_{S_{t_{i+1}},\varphi}.$$
Altogether, we get
$$\cW_{h,\varphi}=\cW_{S_{t_0},\varphi}\onto \cW_{S_{t_1},\varphi}\onto \cdots \onto \cW_{S_{t_{n+1}},\varphi}=\cW_{S,\varphi}.$$
This sequence of epimorphisms naturally defines a sequence of epimorphisms
\begin{equation}\label{=EpiSeq}
\pi_{h,\varphi}=\pi_{S_{t_0},\varphi}\onto \pi_{S_{t_1},\varphi}\onto \cdots \onto \pi_{S_{t_{n+1}},\varphi}=\pi_{S,\varphi}.
\end{equation}
We see that for each $i$, $\pi_{S_{t_{i+1}},\varphi}$ is the quotient of $\pi_{S_{t_i},\varphi}$ by the group $A_i:=L_{t_{i+1}}/R_t$, that we show to be commutative. By Proposition \ref{prop:AV0}  and \cite[Lemma 5.10]{BZ} discussed above, in order to prove  the theorem it is enough to show that   $\pi_{S_{t_i},\varphi}$ is a non-generic representation of $A_i$. For that purpose we show that every unitary character of $A_i$ is given by some $\varphi'\in\fg^*$ with $ad^*(S_{t_{i+1}})\varphi'=-\varphi'$ such that $\varphi'$ does not lie in the tangent space to $\cO$ at $\varphi$.
We then define a quasi-Whittaker quotient  $\pi_{S_{t_{i+1}},\varphi,\varphi'}$, and show that its dual is the space of $(A_i,\chi_{\varphi'})$-equivariant functionals on $\pi_{S_{t_i},\varphi}$. Then we generalize \eqref{=EpiSeq} to quasi-Whittaker quotients, construct some additional epimorphisms and deduce the vanishing of $\pi_{S_{t_{i+1}},\varphi,\varphi'}$ from the vanishing of $\pi_{\cO'}$ for all $\cO'\neq \cO$ with $\cO\subset \overline{\cO'}$. We find quasi-Whittaker quotients to be an important new notion. For an additional evidence for that see Remark \ref{rem:qWhit}.

To prove Theorem \ref{thm:MaxFin} we show in \S \ref{subsec:MaxFin} that the  action of any subgroup of $\widetilde{ G_{\gamma}}$ isomorphic to a quotient of $\widetilde{\SL_{2}(F)}$ is locally finite. By a corollary from Proposition \ref{prop:AV0}  and \cite[Lemma 5.10]{BZ}  it is enough to show that it is non-generic. To show that we choose an $\sl_2$-triple $(e',h',f')$ in the Lie algebra of such a subgroup and let $\varphi'\in \fg$ denote the nilpotent element given by the Killing form pairing with $f'$.
\DimaS{Consider the deformation $S_t:=h+th'$. For $t>1/2$, $e'$ acts trivially on the Whittaker quotient $\pi_{S_t,\varphi}$. We show that the action of $e'$ commutes with the maps in \eqref{=EpiSeq}, and deduce that $\pi_{S_t,\varphi}$ has no $(e',\varphi')$-equivariant functionals for any $t\geq 0$.}

In \S \ref{subsec:Cusp} we deduce Theorem \ref{thm:cusp} from Theorem \ref{thm:MaxModAction}, by way of contradiction. Namely, for any not $F$-distinguished $\cO$ we find a Whittaker pair $(S,\varphi)$ with $\varphi$ in $\cO$ such that $\pi_{S,\varphi}$ is a quotient of a Jacquet module of $\pi$, and thus vanishes. By
Theorem \ref{thm:MaxModAction} we get $\cO\notin \WS(\pi)$.

\DimaD{In \S \ref{sec:PfAdm} we discuss quasi-admissible orbits.
In \S \ref{subsec:PfA}} we deduce Theorem \ref{thm:adm} from Theorem \ref{thm:MaxFin} in the following way. \DimaB{We first note that $\cW_{\varphi}(\pi)$ is a genuine representation of $\widetilde{M_{\gamma}}$, which by Theorem \ref{thm:MaxFin} has a finite-dimensional irreducible  subrepresentation $\rho$. Then we construct from $\rho$ a finite-dimensional genuine representation of $\widetilde{G_{\gamma}}$ and extend it  trivially to  $\widetilde{G_{\varphi}}$.}
In \S \ref{subsubsec:RelSpecOrb} we state several geometric results from \cite{Nev,Nev2,Noel,Oht,PT} and discuss the connection between the notions of special, admissible and quasi-admissible. We also deduce from Theorem \ref{thm:MaxFin} and from \cite{Mat} that the Whittaker support of minimal representations is also minimal.

\DimaH{In \S \ref{sec:compar} we formulate and prove Theorem \ref{thm:ComparOrbit} that provides information on $\WO(\pi)$ given $\WS(\pi)$. The proof is based on the method of \S \ref{sec:MaxOrb}, including the quasi-Whittaker quotients, and on Theorem \ref{thm:MaxModAction}.
 We deduce that for $\pi\in \Rep^{\infty}(\GL_n(F))$, the set $\WO(\pi)$ is closed under the closure order. \DimaK{We also prove a partial analog for $\SL_n(F)$.}}

\DimaD{In \S \ref{sec:Glob} we formulate
global analogs  of Theorems \ref{thm:adm}-\ref{thm:MaxModAction} and explain how to adapt the proofs from \S \ref{sec:MaxOrb}-\ref{sec:compar} to the global case. For example, \cite[Lemma 5.10]{BZ} is replaced by the Fourier decomposition.
Our exposition follows \cite[Chapter 5]{SH}.
}


\subsection{Acknowledgements}

We thank Joseph Bernstein, Guillaume Bossard, David Ginzburg, Maxim Gurevich, \DimaD{Henrik Gustafsson}, Anthony Joseph, Dihua Jiang, \DimaD{Joseph Hundley, David Kazhdan, Axel Kleinschmidt}, Erez Lapid, Baiying Liu, \DimaD{Daniel Persson}, Gordan Savin, Eitan Sayag, and David Soudry for fruitful discussions\DimaN{, and the anonymous referee for useful remarks}.

D.G. was partially supported by ERC StG grant 637912, and ISF grant 249/17. S.S. was partially supported by Simons Foundation grant 509766.

Part of the work on this paper was done during the visit of the three authors to the Max Planck Institute fur Mathematik in Bonn. We thank the administration of the MPI for the invitation and for wonderful working conditions.
\DimaD{
Another part of this paper was written while two of us participated in the program ``Automorphic forms, mock modular forms and string theory" at the Simons Center for Geometry and Physics, Stony Brook University. We are grateful to  the administration of the SCGP and the organizers of the program for wonderful working conditions and for very interesting and helpful talks and discussions on the topic.
}

\section{Preliminaries}\label{sec:prel}

\subsection{Notation}$\,$
\Dima{Let $F$  be either $\R$ or a finite extension of $\Q_p$ and let $\fg$ be a reductive Lie algebra over $F$.}
We say that an element $S\in \fg$ is \emph{rational semi-simple} if its adjoint action on $\fg$ is diagonalizable with  eigenvalues in $\Q$.
For a rational semi-simple element $S$ and a rational number $r$ we denote by $\fg_r^S$ the $r$-eigenspace of the adjoint action of $S$ and by $\fg_{\geq r}^S$ the sum $\bigoplus_{r'\geq r}\fg_{r'}^S$. We will also use the notation $(\fg^*)_{ r}^S$ and $(\fg^*)_{\geq r}^S$ for the corresponding grading and filtration of the dual Lie algebra $\fg^*$.  For $X\in \fg$ or $X\in \fg^*$ we denote by $\fg_X$ the centralizer of $X$ in $\fg$, and by $G_X$ the centralizer of $X$ in $G$.

If $(f,h,e)$ is an $\sl_2$-triple in $\fg$, we will say that $e$ is a nil-positive element for $h$, $f$ is a nil-negative element for $h$, and $h$ is a neutral element for $e$. For a representation $V$ of $(f,h,e)$ we denote by $V^e$ the space spanned by the highest-weight vectors and by $V^f$ the space spanned by the lowest-weight vectors.

From now on we fix a non-trivial unitary additive character
\begin{equation}\label{=chi}
\chi:\F\to \mathrm{S}^1
\end{equation} such that if $\F=\R$  we have $\chi(x)=\exp(2\pi i x)$ and if $\F$ is non-Archimedean the kernel of $\chi$ is the ring of integers.

\subsection{$\sl_2$-triples}

We will need the following lemma which summarizes several well known facts about $\sl_2$-triples.

\begin{lemma}[{See \cite[\S 11]{Bou} or \cite{KosSl2}}]\label{lem:sl2}
$\,$
\begin{enumerate}[(i)]
\item Any nilpotent element is the nil-positive element of some $\sl_2$-triple in $\fg$.
\item If $h$ has a nil-positive element then $e$ is a nil-positive element for $h$ if and only if $e\in \fg^h_2$ and $ad(e)$ defines a surjection $\fg^h_0\onto \fg^h_2$. The set of nil-positive elements for $h$ is open in $\fg^h_2$. 
\item If $e$ is nilpotent then $h$ is a neutral element for $e$ if and only if  $e\in \fg^h_2$ and $h\in \Im(ad(e))$. All such $h$ are conjugate under $G_e$.
\item If $(f,h,e)$ and $(f',h,e)$ are $\sl_2$-triples then $f=f'$.
\item If $(f,h,e)$ is an $\sl_2$-triple and $Z$ commutes with two of its elements then it commutes also with the third one.
\end{enumerate}
\end{lemma}
It is easy to see that the lemma continues to hold true if we replace the nil-positive elements by nil-negative ones (and $\fg_2^h$ by $\fg_{-2}^h$).

\begin{defn}\label{def:nuet}
We will say that $h\in \fg$ is a neutral element for $\varphi\in \fg^*$ if $h$ has a nil-positive element in $\fg$, $\varphi\in (\fg^*)^h_{-2}$, and the linear map $\fg^h_0\to (\fg^*)^h_{-2}$ given by $x \mapsto ad^*(x)(\varphi)$ is an epimorphism.   We also say that $0\in \fg$ is a neutral element for $0\in \fg^*$.
\end{defn}
\Dima{
Note that if we identify $\fg$ with $\fg^*$ in a $G$-equivariant way and assume $\varphi\neq 0$, this property becomes equivalent to $\varphi$ being a nilnegative element for $h$, or $-h$ being a neutral element for $\varphi$.
}

\subsection{Schwartz induction}\label{subsec:ScInd}

For non-Archimedean $F$ we will work with  $l$-groups, {\it i.e.} Hausdorff topological groups that have a basis for the topology at the identity consisting of open compact subgroups. This generality includes $F$-points of algebraic groups defined over $F$, and their finite covers (see \cite[\DimaN{\S 1}]{BZ}).

For $F=\R$ we will work with affine Nash groups, {\it i.e.} groups that are given in $\R^n$ by semi-algebraic equations, and so is the graph of the multiplication map.  This generality includes $\R$-points of algebraic groups defined over $\R$, and their finite covers (see \DimaN{\cite[\S 1.1]{dCl}, \cite[\S 2,3]{AG}, \cite[\S 3]{Sun}, \cite{FS}}).

\Dima{
\begin{notn}\label{not:rep}
If $G$ is an $l$-group, we denote by $\Rep^{\infty}(G)$ the category of smooth representations of $G$ in complex vector spaces. For $V,W\in \Rep^{\infty}(G)$, $V\otimes W$ will denote the usual tensor product over $\C$ and $V^*$ will denote the linear dual.

If $G$ is an affine Nash group, we denote by $\Rep^{\infty}(G)$ the category of smooth nuclear \Fre representations of $G$ of moderate growth. This is essentially the same definition as in \cite[\S 1.4]{dCl} with the additional assumption that the representation spaces are nuclear (see e.g. \cite[\S 50]{Tre}). For $V,W\in \Rep^{\infty}(G)$, $V\otimes W$ will denote the completed projective tensor product and $V^*$  will denote the continuous linear dual, endowed with the strong dual topology.
\end{notn}}

\begin{defn}\label{def:ind}
If $G$ is an $l$-group, $H\subset G$ a closed subgroup and $\pi\in \Rep^{\infty}(H)$, we denote by $\ind_H^G(\pi)$ the smooth compactly-supported induction as in \cite[\S 2.22]{BZ}.

If $G$ is an affine Nash group, $H\subset G$ a closed  Nash subgroup and $\pi\in \Rep^{\infty}(H),$ we denote by $\ind_H^G(\pi)$ the Schwartz induction as in \cite[\S 2]{dCl}. More precisely, in \cite{dCl} du Cloux considers the space  $\Sc(G,\pi)$ of Schwartz functions from $G$ to the underlying space of $\pi,$ and defines a map from  $\Sc(G,\pi)$  to the space $C^{\infty}(G,\pi)$ of all smooth $\pi$-valued functions on $G$ by $f\mapsto \overline f$, where \begin{equation}\label{=ind}
\overline f (x)=\int_{h\in H}\pi(h)f(xh)dh,
\end{equation}
and $dh$ denotes a fixed left-invariant measure on $H$. The Schwartz induction $\ind_H^G(\pi)$ is defined to be the image of this map. Note that $\ind_H^G(\pi)\in \Rep^{\infty}(G)$.
\end{defn}

From now till the end of the subsection let $G$ be either an $l$-group or an affine \Dima{Nash} group, and $H'\subset H \subset G$ be  closed \Dima{(Nash)} subgroups.

\begin{lem}[{\cite[Proposition 2.25(b)]{BZ} and \cite[Lemma 2.1.6]{dCl}}]\label{lem:IndSt}
For any $\pi\in \Rep^{\infty}(H')$ we have
$$\ind_{H'}^G(\pi)\simeq \ind_{H}^G\ind_{H'}^{H}(\pi).$$
\end{lem}

\begin{cor}
For any $\pi\in \Rep^{\infty}(H)$ we have
a natural epimorphism $\ind_{H'}^G(\pi|_{H'})\onto \ind_{H}^G(\pi)$.
\end{cor}

\begin{defn}
For $\pi\in \Rep^{\infty}(G)$, denote by $\pi_G$ the space of coinvariants, i.e. quotient of $\pi$ by the intersection of kernels of all $G$-invariant functionals. Explicitly,
$$\pi_G=\pi/\overline{\{\pi(g)v -v\, \vert \,v\in \pi, \, g\in G\}}, $$
where the closure is needed only for $F=\R$.
\end{defn}
Note that if $F=\R$ and $G$ is connected then $\pi_G=\pi/\overline{\fg_{\C}\pi}$ which in turn is equal to the quotient of $\oH_0(\fg,\pi)$ by the closure of zero.

\begin{lem}\label{lem:Frob2} Let $\rho \in \Rep^{\infty}(H)$, $\pi\in \Rep^{\infty}(G)$.
Consider the diagonal actions of $H$ on $(\pi\otimes \rho)$ and of $G$ on $\pi \otimes \ind_H^G(\rho)$. Then
$(\pi\otimes \rho\Delta_H\Delta_G^{-1})_{H}=(\pi \otimes \ind_H^G(\rho))_{G}$.
\end{lem}
\DimaT{
\begin{proof}
It is easy to see (cf.  \cite[Appendix A]{GGS}) that the integration map $f\mapsto \bar f$ as in \eqref{=ind} gives the following natural isomorphism of representations of $G$, where $G$ acts on functions by left shifts, and coinvariants are taken under the diagonal action on the representation and by right multiplication on the argument: $f^h(x):=\rho(h) f(xh)$:
 \begin{equation}
 \Sc(G,\rho\otimes \Delta_H)_H\cong \ind_H^G(\rho).
 \end{equation}
Substituting in this formula $G$ itself as the subgroup $H$,  we obtain
$\Sc(G,\pi)_G\cong \pi\otimes \Delta_G^{-1}$.

We also have a natural isomorphism of representations of $G\times G\times G$
\begin{equation}
\Sc(G,\pi)\cong \Sc(G)\otimes \pi
\end{equation}
In the non-Archimedean case this is evident, and in the Archimedean case this is \cite[Proposition 1.2.6]{dCl}.
We will also use the linear automorphism of $\Sc(G,\pi)$ given by
\begin{equation}
T(f)(x):=\pi(x)f(x)
\end{equation}
Note that $T$ intertwines the action of $G\times G$ on $\Sc(G,\pi)$ given by $f^{(g_1,g_2)}(x):=\pi(g_2)f(g_1^{-1}xg_2)$ with the action given by $f^{(g_1,g_2)}(x):=\pi(g_1)f(g_1^{-1}xg_2)$. Altogether we have
\begin{equation*}
(\pi \otimes \ind_H^G(\rho))_{G}\cong (\pi\otimes \Sc(G)\otimes \rho\otimes \Delta_H)_{G\times H}\cong (\Sc(G,\pi)\otimes \rho\otimes \Delta_H)_{ G\ \times H}\cong (\pi\otimes \rho\otimes \Delta_H\otimes \Delta_G^{-1})_{H}
\end{equation*}
\end{proof}
}



\subsection{Oscillator representations of the Heisenberg group}
\begin{defn}
Let $W_n$ denote the  2n-dimensional $\F$-vector space $(\F^n)^* \oplus \F^n$ and let $\omega$ be the standard symplectic form on $W_n$.
The Heisenberg group $H_n$ is the algebraic group with underlying algebraic variety $W_{n} \times \F$ with the group law given by $$(w_1,z_1)(w_2,z_2) = (w_1+w_2,z_1+z_2+(1/2)\omega(w_1,w_2)).$$
Note that $H_0=\F$.
\end{defn}


\begin{defn}
Let $\chi$ be the additive character  of $\F$, as in \eqref{=chi}.  Extend $\chi$ trivially to a character of the commutative subgroup $0\oplus \F^n\oplus \F \subset H_n$.
The oscillator representation ${\varpi}_\chi$ is the unitary induction of $\chi$ from $0\oplus \F^n\oplus \F$ to $H_n$. Define the smooth oscillator representation
$\sigma_\chi$ to be the space of smooth vectors in ${\varpi}_\chi$.
\end{defn}

\begin{lemma}[{see {\it e.g.} \cite[Corollary 2.4.5]{GGS}}]\label{lem:OscInd}
$\sigma_{\chi}=\ind_{0\oplus \F^n\oplus \F}^{H_n}(\chi)$
\end{lemma}

%
%

\begin{theorem}[Stone-von-Neumann]\label{thm:StvN}
The oscillator representation $\varpi_\chi$ is the only irreducible unitary representation of $H_n$ with central character $\chi$.
\end{theorem}

\begin{cor}\label{cor:StvN}
 Let $L\subset W$ be a Lagrangian subspace. Extend $\chi$ trivially to the abelian subgroup $L\oplus \F \subset H_n$. Then $\ind_{L\oplus \F}^{H_n}\chi \cong \sigma_\chi$.
\end{cor}
%

\subsection{Degenerate Whittaker models}\label{subsec:DefMod}
Let $G$ be \DimaE{a finite central extension of the group $G^{\mathrm{alg}}$ of $F$-points of a reductive algebraic group defined over $F$. Let $G^{ad}$ denote the corresponding adjoint algebraic group.

\begin{lem}[{\cite[Appendix I]{MW_Scr}}]\label{lem:cov}
Let $U\subset G^{\mathrm{alg}}$ be a unipotent subgroup, and $\hat U$ be the preimage of $U$ in $G$. Then there exists a unique open subgroup $U'\subset \hat U$ that projects isomorphically onto $U$.
\end{lem}

We will therefore identify the unipotent subgroups of $G^{\mathrm{alg}}$ with their liftings in $G$.}

\begin{defn}
\begin{enumerate}[(i)]
\item A \emph{Whittaker pair} is an ordered pair $(S,\varphi)$ such that $S\in \fg$ is rational semi-simple, and $\varphi\in (\fg^*)^{S}_{-2}$. Given such a Whittaker pair, we define the space of \emph{degenerate Whittaker models} $\cW_{S,\varphi}$ in the following way: let $\fu:=\fg_{\geq 1}^S$.
Define an anti-symmetric form $\omega_\varphi$ on $\fg$ by $\omega_\varphi(X,Y):= \varphi([X,Y])$. Let $\fn$ be the radical of $\omega_\varphi|_{\fu}$. Note that $\fu,\fn$ are nilpotent subalgebras of $\fg$, and $[\fu,\fu]\subset \fg^S_{\geq 2}\subset \fn$.
 Let $U:=\Exp(\fu)$ and $N:=\Exp(\fn)$ be the corresponding nilpotent subgroups of $G$.
Let  $\fn' :=\fn\cap \Ker(\varphi), \, N':=\Exp(\fn')$. If $\varphi=0$ we define \begin{equation}\cW_{S,0}:=\ind_{U}^G(\C).\end{equation}
Assume now that $\varphi$ is non-zero.
Then $U/N'$ has a natural structure of a Heisenberg group, and its center is $N/N'$. Let $\chi_{\varphi}$ denote the unitary character of $N/N'$ given by $\chi_{\varphi}(\exp(X)):=\chi(\varphi(X))$.
Let $\sigma_\varphi$ denote the oscillator representation of $U/N'$ with central character $\chi_{\varphi}$, and $\sigma'_\varphi$ denote its trivial lifting to $U$. Define \begin{equation}\cW_{S,\varphi}:=\ind_{U}^G(\sigma'_\varphi).\end{equation}


\item For a nilpotent element $\varphi\in \fg^*$, define the \emph{generalized Whittaker model} $\cW_{\varphi}$  corresponding to $\varphi$ to be $\cW_{S,\varphi}$, where $S$ is a neutral element for $\varphi$ if $\varphi\neq 0$ and $S=0$ if $\varphi=0$. We will also call $\cW_{S,\varphi}$ \emph{neutral degenerate Whittaker model}. By Lemma \ref{lem:sl2} $\cW_{\varphi}$ depends only on the coadjoint orbit of $\varphi$, and does not depend on the choice of $S$. Thus we will also use the notation $\cW_{\cO}$ for a nilpotent coadjoint orbit $\cO\subset \fg^*$. See \cite[\S 5]{GGS} for a formulation of this definition without choosing $S$.

\item \DimaG{To} $\pi\in \Dima{\Rep^{\infty}}(G)$ \DimaG{associate} the degenerate and generalized Whittaker quotients  by
\begin{equation}\pi_{S,\varphi}:=(\cW_{S,\varphi}\otimes\pi)_G \text{ and }\pi_{\varphi}:=(\cW_{\varphi}\otimes\pi)_G.
\end{equation}
\end{enumerate}
\end{defn}

\begin{lem}\label{lem:WhitFrob}
Let $\fl\subset \Dima{\fg_{\geq 1}^S}$ be a maximal isotropic subalgebra and $L:=\Exp(\fl)$. Let $\pi\in \DimaD{\Rep^{\infty}}(G)$. Then
$$\pi_{S,\varphi}\cong (\pi\otimes\chi_{\varphi})_L.$$
\end{lem}
\begin{proof}
By Corollary \ref{cor:StvN} and Lemma \ref{lem:IndSt} we have
$\cW_{S,\varphi}\cong \ind_L^G(\chi_{\varphi}).$ Using Lemma \ref{lem:Frob2} we obtain
$$\pi_{S,\varphi}\cong(\ind_L^G(\chi_{\varphi})\otimes\pi)_G\cong (\pi\otimes\chi_{\varphi})_L.$$
\end{proof}

Slightly different degenerate Whittaker models are considered in \cite{GGS} and denoted
$\cW_{S,\varphi}(\pi)$. By Lemma \ref{lem:WhitFrob} and \cite[Lemma 2.5.2]{GGS} they relate to $\pi_{S,\varphi}$ by $\cW_{S,\varphi}(\pi)=\pi_{S,\varphi}^*$.
We changed the notion in this paper since for p-adic $F$, $\pi_{S,\varphi}$ are the models considered in \cite{MW,Var} and for $F=\R$, $\pi_{S,\varphi}$ are (nuclear) \Fre spaces.

\begin{remark} \label{rem:RealGen}
If $\F=\R$, one can
define $\cW_{S,\varphi}$ for any semi-simple $S$ with real eigenvalues in the same way, and all our proofs  will be valid for this case without changes.\end{remark}


\begin{lem}\label{lem:torus}
Assume that $G$ is an adjoint group, and let
$S\in \fg$ be semi-simple. Then there exists an algebraic group morphism  $\nu:\F^{\times}\to G$ (defined over $\F$) with $d\nu(1)=S$ if and only if all the eigenvalues of $S$ in the adjoint action on $\fg(\bar \F)$ are integers, where $\bar F$ denotes the algebraic closure of $F$.
\end{lem}

\begin{proof}
Embed $G$  into $G':=\GL(\fg)$ using the adjoint action. Then any $\nu:\F^\times \to G$ defines $\nu': \F^\times \to G'$. Note that a semi-simple $S'\in \fg'$ equals $d\nu'(1)$ for some $\nu': \F^\times \to G'$ if and only if all the eigenvalues of $S'$ on $\fg(\bar \F)$ are integers. The ``only if'' part follows. For the ``if'' part, note that if $S=d\nu'(1)$ then $\Im(d\nu')\subset \fg$, hence $\Im(\nu')\in G$ and thus $\nu'$ defines $\nu:\F^{\times}\to G$ with the required property.
\end{proof}


\begin{cor}\label{cor:tor}
Let $S\in \fg$ be rational semi-simple. Then there exists an algebraic group morphism  $\nu:\F^{\times}\to \DimaE{G^{ad}}$ and a central element $Z\in \fg$ such that $S-Z\in \Im(d\nu)$.
\end{cor}
\begin{proof}
Replacing $S$ by an integer multiple we can assume that all the eigenvalues of $S$ in the adjoint action on $\fg(\bar \F)$ are integers.
Thus there exists an algebraic group morphism from $F^{\times}$ to the adjoint group of $G$ which includes the projection of $S$ in its image.
\end{proof}

\subsection{Covering groups}\label{subsec:cov}

Let $\gamma=(e,h,f)$ be an $\sl_2$-triple in $\fg$.
 Let $G_{\gamma}$ denote the joint centralizer of the three elements of $\gamma$. It is well known that $G_{\gamma}$ is a Levi subgroup of $G_{f}$.
Let  $\varphi \in \g^{\ast}$ be given by the Killing form pairing with $f$. Recall that  it induces a nondegenerate symplectic form $\omega_{\varphi}$ on $\g^h_1$ and note that $G_{\gamma}$ acts on $\g^h_1$ preserving the symplectic form. That is, there is a natural map $G_{\gamma}\rightarrow \Sp(\g^h_1)=\Sp(\omega_{\varphi})$. Let $\widetilde{\Sp(\omega_{\varphi})}\rightarrow \Sp(\omega_{\varphi})$ be the metaplectic double covering, and set
\[
 \widetilde{G_{\gamma}}=G_{\gamma}\times_{\Sp(\omega_{\varphi})} \widetilde{\Sp(\omega_{\varphi})}.
\]
Observe that the natural map $\widetilde{G_{\gamma}}\rightarrow G_{\gamma}$ defines a double cover of $G_{\gamma}$. \DimaA{We denote by $M_{\gamma}$  the  subgroup of $G_{\gamma}$  generated by the unipotent elements.}
Let $\widetilde{M_{\gamma}}$ denote the preimage of  $M_{\gamma}$ under the projection $\widetilde{G_{\gamma}}\to G_{\gamma}$. Note that  different choices of $\gamma$ with the same $f$ lead to conjugate groups $G_{\gamma}$ and $M_{\gamma}$.

One can also define a covering $\widetilde{G_{\varphi}}$ of the group $G_{\varphi}=G_f$, using the symplectic form defined by $\varphi$ on $\fg/\fg^{\varphi}$. It is easy to see that this cover splits over the unipotent radical of $G_{\varphi}$, and that the preimage of $G_{\gamma}$ in $\widetilde{G_{\varphi}}$ is isomorphic to $\widetilde{G_{\gamma}}$, see {\emph e.g.} \cite[\S 4]{Nev}.

\begin{defn}
Let $\bf H$ be a linear algebraic group defined over $F$, and
fix an embedding $\bf H \into \GL_n$. Denote by $H_0$ the open normal subgroup of ${\bf H}(F)$
generated by the image of the exponential map.
\end{defn}
Note that $H_0$ does not depend on the embedding of $\bf H$ into $\GL_n$. Note also that if $\bf H$ is semi-simple \DimaO{and connected} then $H_0={\bf H}(F)$ and if $F=\R$ then $H_0$ is the connected component of ${\bf H}(F)$. For a \DimaE{ finite central extension $H'$ of $H$, we  define $H'_0$ to be the preimage of $H_0$ under the projection $H'\onto H$.}

\begin{definition}[{\cite[\S 4]{Nev}}]
 We say that a nilpotent orbit $\Orb\subset \g^{\ast}$ is \emph{admissible} if for some (equivalently, for any) choice of  $\varphi \in \cO$, the covering $\widetilde{G_{\varphi}}\rightarrow G_{\varphi}$ splits over $(G_{\varphi})_0$.
\end{definition}

As observed in \cite{Nev}, this definition of admissibility is compatible with Duflo's original definition for the Archimedean case, given in  \cite[\S II.2]{D}.

\DimaA{
\begin{definition}
 We say that a nilpotent orbit $\Orb\subset \g^{\ast}$ is \emph{quasi-admissible} if for some (equivalently, for any) 
 $\varphi \in \cO$, the covering $\widetilde{G_{\varphi}}\rightarrow G_{\varphi}$ admits a finite dimensional \emph{genuine} representation, that is, a finite dimensional representation on which the
  non-trivial element $\eps$ in the preimage of $1\in G_{\varphi}$ acts by $-\Id$. 
\end{definition}

%
}

\begin{defn}\label{def:act}
\DimaD{Let us define the action of $\widetilde{G_{\gamma}}$ on $\pi_{\varphi}$.
Since the adjoint action of $G_{\gamma}$ preserves $\fg^h_1$ and the symplectic form on it, it preserves $U/N'$. Since $\sigma_{\varphi}$ is the unique smooth irreducible representation of $U/N'$ with central character $\chi_{\varphi}$, we have a projective action of $G_{\gamma}$ on $\sigma_{\varphi}$. By \cite{Weil} this action lifts to a genuine representation of $\widetilde{G_{\gamma}}$. This gives \DimaS{rise} to an action of $\widetilde{G_{\gamma}}$ on $\cW_{\varphi}$
\DimaG{by $(\tilde g f)(x)=\tilde g( f(x\DimaS{g}))$. This action  commutes with the action of $G$ and thus} defines an action of $\widetilde{G_{\gamma}}$ on $\pi_{\varphi}=(\cW_{\varphi}\otimes \pi)_G.$
}

\DimaS{More generally, for a Whittaker pair $(S,\varphi)$ let $G(S,\varphi)$ denote the subgroup of $G$ corresponding to the Lie algebra $(\fg^{S}_{\geq 0}\cap \fg_{\varphi})\oplus \fg^{S}_{> 1}$. Then $\DimaT{\widetilde{G(S,\varphi)}}$ normalizes the groups $U$ and $N'$ corresponding to the pair $(S,\varphi)$ and acts on $\cW_{S,\varphi}$ and on $\pi_{S,\varphi}$ as above.}
\end{defn}

\DimaS{It is well known that the metaplectic cover splits over the Siegel parabolic corresponding to any Lagrangian. Indeed, the Weil representation on the Siegel parabolic is given by the action on the Heisenberg group by automorphisms (in the realization of the representation functions on the Heisenberg group as in Lemma \ref{lem:OscInd}). These considerations imply the following lemma.

\begin{lem}\label{lem:actions}Let $(S,\varphi)$ be a Whittaker pair,  let $\fl\subset \Dima{\fg_{\geq 1}^S}$ be a maximal isotropic subalgebra and $L:=\Exp(\fl)$. Let $G'$ denote the normalizer of $L$ in $G(S,\varphi)$.   Then
\begin{enumerate}[(i)]
\item The metaplectic cover of $G(S,\varphi)$ splits over $G'$.
\item The action of $G'$ on $\cW_{S,\varphi}$ defined by splitting the metaplectic cover corresponds under the isomorphism $\cW_{S,\varphi}\cong\ind_{L}^G(\chi_\varphi)$
to the action $(gf)(x):=f(xg)$.
\item For any  $\pi\in \DimaD{\Rep^{\infty}}(G)$, the action of $G'$ on $\pi_{S,\varphi}$ defined by splitting the metaplectic cover corresponds under the isomorphism
$$\pi_{S,\varphi}\cong (\pi\otimes\chi_{\varphi})_L$$
from Lemma \ref{lem:WhitFrob} to the action on $(\pi\otimes\chi_{\varphi})_L$ given by action on representatives in $\pi$.
\end{enumerate}
\end{lem}
}


%

%

\section{Some Archimedean technical lemmas}\label{sec:Arch}
\setcounter{lemma}{0}

Let $P_2(\R)$ denote the group of affine transformations of the line.
Let $N$ denote the unipotent radical of $P_2(\R)$ and $\fn$ denote the Lie algebra of $N$. Let $B^0$ denote the connected component of the identity in $P_2(\R)$

\begin{prop}\label{prop:AV0}
Let $V\in\Rep^{\infty}(B^0)$.
Suppose that $V$ is not generic, i.e. $(V^*)^{N,\psi}=0$ for any non-trivial unitary character $\psi$ of $N$. Then $\fn$ acts locally nilpotently on $V^*$.
\end{prop}

In the p-adic case, an analogous lemma is proven by Bernstein and Zelevinsky using l-sheaves. We will prove this proposition in \S \ref{subsec:PfAV0} using \cite[\S 2]{dCl}. Let us now derive some corollaries.

\begin{cor}\label{cor:ExFun}
Let $V\in\Rep^{\infty}(B^{0})$ be non-zero.
Then $(V^*)^{N,\psi}\neq0$ for some (possibly trivial) unitary character $\psi$ of $N$.
\end{cor}

Let $Mp$ denote the metaplectic group $\widetilde{\SL_2(\R)}$, and $\fs:=\Lie(Mp)=\sl_2(\R)$ denote its Lie algebra.

%

\begin{cor}\label{cor:FinQuot}
Let $V\in \Rep^{\infty}(Mp)$ be non-zero and non-generic. Then any $\lam \in V^*$ generates a finite-dimensional subrepresentation, and $V$ has a (non-zero) finite-dimensional quotient. In particular, every irreducible genuine $V\in \Rep^{\infty}(Mp)$ is generic.
\end{cor}
\begin{proof}
Choose a standard basis $e,h,f$ for $\fs$.
Let $B$ denote a Borel subgroup of $\SL_2(\R)$ and $\widetilde B\subset Mp$ denote its preimage. Note that the connected component of $\widetilde B$ is isomorphic to $B^0$ as a Nash group. Thus Proposition \ref{prop:AV0} implies that both $e$ and $f$ act locally nilpotently on $V^*$. By e.g. \cite[Lemma C.0.3]{AG_HC} this implies that any $\lam \in V^*$ generates a finite-dimensional subrepresentation $W$. This $W$ non-degenerately pairs with a quotient  of $V$, and thus this quotient is finite-dimensional. \end{proof}
\DimaP{
\begin{remark}
Under the assumption of the corollary we cannot in general claim that the action on $V$ is locally finite. For example, $V$ could be the direct product of all irreducible finite-dimensional representations of $\SL_2(\R)$, with the topology given by projections.
\end{remark}


In order to apply Corollary \ref{cor:FinQuot} we will need the following lemma\DimaQ{s.

\begin{lem}\label{lem:Finsl2}
Let  $\fm$ be a Lie algebra generated by ad-nilpotent elements.
Then $\fm$ has a basis consisting of ad-nilpotent elements.
\end{lem}
\begin{proof}
Let $\fr\subset \fm$ denote the subspace spanned by ad-nilpotent elements. Then $\fr$ is a module for the Lie group corresponding to $\fm$, and thus a subalgebra (in fact, an ideal) of $\fm$. Since $\fr$ generates $\fm$, we obtain $\fr=\fm$.
\end{proof}
Note that if $\fm $ is semi-simple then ad-nilpotent is the same as nilpotent.

\begin{lem}\label{lem:FinPBW}
Let $\fm$ be a Lie algebra and $V$ be an $\fm$-module. Suppose that $\fm$ has a basis $X_1,\dots X_n$ such that each $X_i$ acts locally finitely on $V$. Then $\fm$ acts locally finitely on $V$.
\end{lem}
\begin{proof}
Let $\mathcal{U}$ denote the universal enveloping algebra of $\fm$.
For each $i$, let $p_i$ denote the subspace of $\mathcal{U}$ spanned by all the powers of $X_i$. By the Poincar\'e-Birkoff-Witt theorem, $\mathcal{U}=p_np_{n-1}\dots p_1$.
By the assumption of the lemma, for any finite-dimensional subspace $W\subset V$, and any $i$, $p_iW$ is finite-dimensional. By induction we obtain that for any $v\in V$, and $k\leq n$, $p_kp_{k-1}\dots p_1v$ is finite-dimensional. Taking $k=n$ we obtain that $\mathcal{U}v$ is finite-dimensional.
\end{proof}
}

}

\subsection{Proof of Proposition \ref{prop:AV0}}\label{subsec:PfAV0}

\begin{lemma}
Let $V$ be a \Fre module over the \Fre algebra of power series $\C[[t]]$. Then $t$ acts locally nilpotently on the dual space $V^*$.
\end{lemma}
\begin{proof}
Any $\lam\in V^*$ and $v\in V$ define a distribution on $\R$ supported at $0$ by $\xi_{\lam,v}(f):=\lam(\bar f v)$, where $\bar f$ denotes the Taylor series of $f$ at $0$. \Dima{Then, for any $\lam\in V^*$ there exists $k$ such that $t^k\xi_{\lam,v}=0$, and thus $\lam(t^kv)=0$. Thus $V=\cup \Ker(t^k\lam)$. Since $V$ is metrizable, the Baire category theorem implies that $\Ker(t^k\lam)=V$ for some $k$.}
\end{proof}

\begin{cor}
Consider the \Fre algebra of Schwartz functions $\Sc(\R)$, under multiplication, and let $V$ be a \Fre $\Sc(\R)$-module annihilated by the ideal $\Sc(\R\setminus \{0\})$. Let $f\in \Sc(\R)$ with $f(0)=0$. Then $f$ acts locally nilpotently on $V^*$.
\end{cor}

Using Fourier transform we obtain the following corollary.

\begin{cor}\label{cor:LocNilp}
Let $V\in\Rep^{\infty}(\R)$. Let
\begin{equation}\label{=A}
A:=\{f \in \Sc(\R) \text{ s.t. }\forall k. \, \int_{-\infty}^{\infty} t^k f(t) dt =0\}.
\end{equation}
Suppose that $AV=0$. Then $\frac{d}{dt}$ acts locally nilpotently on $V^*$.
\end{cor}

\Dima{
\begin{defn}[{\cite[\S 2]{dCl}}]
Let $G$ be an affine Nash group and $X$ be an affine  Nash $G$-space.
A \emph{$(G,\Sc(X))$-module} is a representation $E\in \Rep^{\infty}(G)$ with a continuous action $\pi$ of the \Fre algebra $\Sc(X)$ satisfying $g\pi(f)g^{-1}=\pi(f^g)v$, where $f^g(x):=f(g^{-1}x)$. We say that $E$ is \emph{non-degenerate} if the action map $\Sc(X)\otimes E \to E$ has dense image, and for every $v\neq 0 \in E$, there exists $f\in \Sc(X)$ with $\pi(f)v\neq 0$.

We denote the category of $(G,\Sc(X))$-modules by $\Rep_{G,X}$ and the subcategory of non-degenerate modules by $\Rep_{G,X}^{\mathrm{nd}}$.
\end{defn}

\begin{thm}[{\cite[Lemma 2.5.7 and Theorem 2.5.8]{dCl}}]\label{thm:dCl}
Let $G$ be an affine Nash group and $X$ be a transitive Nash $G$-space. Define $\tilde \Sc(X):=\C\oplus \Sc(X)$. Fix $x_0\in X$ and let $m_{x_0}\subset \tilde \Sc(X)$ denote the maximal ideal consisting of functions vanishing at $x_0$. Let $E\in \Rep_{G,X}^{\mathrm{nd}}$. Then $m_{x_0}E$ is a closed proper subspace of $E$.
\end{thm}
}
\begin{proof}[Proof of Proposition \ref{prop:AV0}]
 Let the algebra $A$ from \eqref{=A} act on $V$ using the identification $N\simeq \R$. By Corollary \ref{cor:LocNilp} it is enough to show that $AV=0$.
Suppose by way of contradiction that $AV\neq 0$ and let $E$ denote the closure of $AV$. Note that Fourier transform defines an isomorphism $A\simeq \Sc(\R \setminus \{0\})$. We further identify $\R \setminus \{0\}$ with the Nash manifold $X$ of non-trivial unitary characters of $N$.
The action of $P_2(\R)$ on $X$ is by multiplication by the reductive part, and the action of $B^0$ has two orbits: $X_+$ and $X_-$

We let $\Sc(X)$ act on $E$ through $A$, and note that this is compatible with the action of the group $B^0$.
 Theorem \ref{thm:dCl} implies that either $(E^*)^{N,\psi}\neq 0$ for any $\psi \in X_+$ or $(E^*)^{N,\psi}\neq 0$ for any $\psi \in X_-$. Fix  a $\psi$ with $(E^*)^{N,\psi}\neq 0$.
From the short exact sequence
$$ 0 \to (V/E)^* \to V^* \to E^* \to 0$$ we obtain the exact sequence
$$(V^*)^{(N,\psi)}\to (E^*)^{(N,\psi)}\to \oH^1(\fn,(V/E)^*\otimes \psi).$$
Note  that $\fn$ acts locally nilpotently on $(V/E)^*$ by Corollary \ref{cor:LocNilp}.
Thus the action of any generator of $\fn$ on $(V/E)^*\otimes \psi$ is invertible and hence  $H^1(\fn,(V/E)^*\otimes \psi)=0$.
Since $(E^*)^{N,\psi}\neq 0$
we obtain that $(V^*)^{(N,\psi)}\neq 0$, which contradicts the conditions of the proposition. Thus $AV=0$ and thus $\fn$ acts locally nilpotently on $V^*$.
\end{proof}

\section{Proof of Theorem \ref{thm:MaxModAction}}\label{sec:MaxOrb}
\setcounter{lemma}{0}

\Dima{Fix a Whittaker pair $(S,\varphi)$.} Let $G \cdot \varphi$ denote the coadjoint orbit of $\varphi$ and $\overline{G \cdot \varphi}$ denote its closure.

\begin{lem}[{\cite[Lemma 3.0.2]{GGS}}]\label{lem:Z}
    \label{lem:Ztogeneralized}
There exists $Z\in \fg^{S}_0$ such that $(S-Z,\varphi)$  is a neutral Whittaker pair.
\end{lem}

\Dima{Fix $Z$  as in the lemma and let $h:=S-Z$}. For any rational number $t \geq 0$ define \begin{equation}\label{=ut}
S_t:=\Dima{S_{t,Z}:=}S+(t-1)Z,\quad \fu_t:=\fg^{S_t}_{\geq 1},\quad \fv_t:=\fg^{S_t}_{> 1},\text{ and }\fw_t:=\fg^{S_t}_{1}. \quad
\end{equation}
\begin{defn}\label{def:crit}
We call $t$ \emph{regular} if $\fu_t = \fu_{t+\eps}$ for any small enough $\eps\in \Q$. \Dima{Observe that this is equivalent to } $\fw_t\subset \fg_Z$. If $t$ is not regular we call it \emph{critical}. For convenience, we will say that 0 is critical.
\end{defn}
Note that \Dima{for any $Z$ and $\varphi$ }there are only finitely many critical numbers.
Recall the anti-symmetric form $\omega$ on $\fg$ given by $\omega(X,Y)=\varphi([X,Y])$.

\begin{lemma}[{\cite[Lemma 3.2.6]{GGS}}]\label{lem:help}$\,$
\label{lem:kernel}
\begin{enumerate}[(i)]
\item \label{it:OmInv} The form $\omega$ is $ad(Z)$-invariant.
\item \label{it:KerOm} The radical of $\omega$ is $\Rad \omega = \fg_\varphi=\fg^f\subset \fg^h_{\leq 0}$.
\item \label{it:KerNeg} $\Rad(\omega|_{\fw_t})=\Rad(\omega)\cap \fw_t$. 
\item \label{it:Kerv} $\Rad(\omega|_{\fu_t})=\fv_t\oplus \Rad(\omega|_{\fw_t}) $.
\item \label{it:LW}  $\fw_t\cap\fg_\varphi\subset \fu_{T}$  for any $t<T$.
\end{enumerate}
\end{lemma}
Choose a Lagrangian $\fm\subset \fg^Z_0\cap \fg^S_{1}$
 and let
 \begin{equation}\label{=lt}
 \fl_t:=\fm+(\fu_t\cap \fg^{Z}_{< 0})+\Rad(\omega|_{\fu_t})\text{ and }\fr_t:=\fm+(\fu_t\cap \fg^{Z}_{> 0})+\Rad(\omega|_{\fu_t}).
\end{equation}


Note that these are maximally isotropic subspaces.

\begin{lemma}\label{lem:lrw}
Let $0\leq t<T$ and suppose that there are no critical numbers in $(t,T)$. Then
\begin{equation}\label{=lrw}
\fl_{T}=\fr_{t}\oplus (\fw_{T}\cap \fg_{\varphi}).
\end{equation}
Moreover, $\fr_{t}$ is an ideal in $\fl_{T}$ with commutative quotient and $\fv_{T}$ is an ideal in $\fr_{t}$ with commutative quotient.
\end{lemma}
\begin{proof}
Decomposing to the eigenspaces of $ad(Z)$ we obtain
\begin{align}
\fr_{t}&=\fm \oplus (\fw_{t}\cap \fg^Z_{>0})\oplus (\fv_{t}\cap \fg^Z_{\geq0}) \oplus (\fv_{t}\cap \fg^Z_{<0}) \\
\fl_{T}&= \fm \oplus (\fw_{T}\cap \fg^Z_{<0})\oplus (\fv_{T}\cap \fg^Z_{<0}) \oplus (\fv_{T}\cap \fg^Z_{\geq0}) \oplus (\fw_{T}\cap \fg_{\varphi})
\end{align}
Since there are no critical numbers in $(t,T)$ we have
\begin{align}
\fv_{t}\cap \fg^Z_{<0}&=(\fw_{T}\cap \fg^Z_{<0})\oplus (\fv_{T}\cap \fg^Z_{<0})\\
\fv_{T}\cap\fg^Z_{\geq0}&=(\fw_{t}\cap \fg^Z_{>0})\oplus (\fv_{t}\cap \fg^Z_{\geq0})
\end{align}
This implies $\fl_{T}=\fr_{t}\oplus (\fw_{T}\cap \fg_{\varphi})$ and $\fv_T\subset \fr_t$. The rest is straightforward.
\end{proof}

\subsection{Basic comparison lemmas}\label{subsec:basic}

\begin{lemma}[{\cite[Lemma 5.10]{BZ}}]\label{lem:BZ}
Assume that $F\neq \R$ and let $A$ be a finite-dimensional vector space over $F$, viewed as an $l$-group. Let $\rho$ be a smooth representation of $A$. Suppose that $\Hom(\rho,\chi)=0$ for every non-trivial smooth character $\chi$ of $A$. Then $\rho$ is a trivial representation.
\end{lemma}

Our main tools are Lemma \ref{lem:BZ}, Proposition \ref{prop:AV0}
and the following notion \DimaO{and lemma}.

\begin{defn}
\Dima{We say that $(S,\varphi,\varphi')$ is a \emph{Whittaker triple} if $(S,\varphi)$ is a Whittaker pair and }
$\varphi'\in (\fg^*)^S_{>-2}$.

For a Whittaker triple $(S,\varphi,\varphi')$ we define a smooth representation of $G$ that we call a \emph{quasi-Whittaker model} and denote $\cW_{S,\varphi,\varphi'}$, in the following way.
Let
$$\fu:=\fg_{\geq 1}^S, \, \fv:=\fg_{> 1}^S, \, \fz:=\fv\oplus\Dima{((\fg_{1}^S)\cap \fg_{\varphi})} \text{ and } \fk \text{ be the kernel of }\varphi+\varphi' \text{ on } \fz.$$
Then $\Exp(\fu)/\Exp(\fk)$ is a Heisenberg group with center $\Exp(\fz)/\Exp(\fk)$, and $\varphi+\varphi'$ is a character of this center. Let $\sigma_{\varphi,\varphi'}$ denote the oscillator representation corresponding to this character. Continue $\sigma_{\varphi,\varphi'}$ trivially to $\Exp(\fu)$ and let \Dima{$\cW_{S,\varphi,\varphi'}:=\ind_{\Exp(\fu)}^G\sigma_{\varphi,\varphi'}$ denote its Schwartz induction to $G$ (see Definition \ref{def:ind}). }

Note that for $\varphi'=0$ we obtain the degenerate Whittaker model $\cW_{S,\varphi}$. Note also that $\varphi'$ vanishes on $[\fu,\fu]$ and thus $\varphi+\varphi'$ defines the same anti-symmetric form on $\fu$ as $\varphi$.

\begin{remark}\label{rem:qWhit}
Let $a$ be the first eigenvalue of $S$ bigger than 1. Then
$$\ind_{\Exp(\fv)}^G\chi_{\varphi+\varphi'}=\cW_{a^{-1}S,0,\varphi+\varphi'}.$$
\end{remark}

For $\pi \in \Rep^\infty(G)$ define $\pi_{S,\varphi,\varphi'}:=(\cW_{S,\varphi,\varphi'}\otimes\pi)_G$.
\Dima{Recall that if $F=\R$\ then $\otimes$ denotes the completed tensor product.} We will say that $\pi$ is $(S,\varphi,\varphi')$-distinguished if $\pi_{S,\varphi,\varphi'}\neq 0$. We will denote by $\QWO(\pi)$ the set of all orbits $\cO$ for which there exists a Whittaker triple $(S,\varphi,\varphi')$ such that $\varphi\in \cO$ and $\pi$ is $(S,\varphi,\varphi')$-distinguished. The set of maximal orbits in $\QWO(\pi)$ will be denoted $\QWS(\pi)$.
\end{defn}


Till the end of the subsection we let $T>t\geq 0$ be such that there are no critical numbers in $(t,T)$. We also
 \Dima{fix} $\varphi'\in (\fg^*)^{S_t}_{>-2}\cap (\fg^*)^{S_T}_{>-2}$ and $\psi \in (\fg^*)^{S_t}_{>-2}\cap (\fg^*)^{S_T}_{-2}$. Let $L_{t}:=\Exp(\fl_t),\, R_{t}:=\Exp(\fr_t)$ and let $\chi:=\chi_{\varphi+\varphi'+\psi}$ be the character of these groups given by $\varphi+\varphi'+\psi$.

Similarly to Lemma \ref{lem:WhitFrob} we have
\begin{lem}\label{lem:Heis}
$$(\pi\otimes \chi_{\varphi+\varphi'})_{L_t}\simeq(\pi\otimes \chi_{\varphi+\varphi'})_{R_t}\simeq\pi_{S_t,\varphi,\varphi'}.$$
\end{lem}

Let $f\in \fg$ be the unique nilpotent element corresponding to $\varphi$ by the Killing form.
Let $h:=S-Z$ and let $\gamma=(e,h,f)$ be an $\sl_2$-triple.

\begin{lem}\label{lem:step}
Assume that for any non-zero $\psi'\in (\fg^*)^{S_T}_{-1}\cap (\fg^*)^e$ we have
 $$\pi_{S_T,\varphi,\varphi'+\psi+\psi'}=0.$$
 Then
 \begin{enumerate}[(a)]
 \item \label{it:padicStep} If $F$ is non-Archimedean then $(\pi\otimes\chi)_{L_{T}}=(\pi\DimaS{\otimes}\chi)_{R_{t}}$.
  In other words, any $(R_{t},\chi)$-equivariant functional on $\pi$ is automatically $(L_{T},\chi)$-equivariant.
\item \label{it:RealStep} If $F=\R$ then the commutative  Lie algebra $\fa:=\fl_T/\fr_t$ acts on $((\pi\otimes\chi)_{R_{t}})^*$ locally nilpotently and \DimaS{thus} if $(\pi\otimes\chi)_{R_{t}}\neq 0$ then $((\pi\otimes\chi)_{R_{t}})_{\fa}=(\pi\otimes\chi)_{L_{T}}\neq 0$.
\end{enumerate}
\end{lem}
\begin{proof}
Let $\rho:=\pi_{R_{t},\chi}$. The quotient $A:=L_{T}/R_{t}$ acts on $\rho$.

\eqref{it:padicStep}
We have to show that the action of $A$ on $\rho$ is trivial,
By Lemmas \ref{lem:BZ} and \ref{lem:lrw}, it is enough to show that for any non-trivial character $\chi'$ of  $A$, $\Hom_A(\rho,\chi')=0$.
By Lemma \ref{lem:lrw}, characters of $A$ are given by elements of $(\fw_{T}\cap \fg_{\varphi})^*\cong(\fw_{T}^*)^e$. For $\psi'\in (\fw_{T}^*)^e$ and the corresponding character $\chi'_{\psi'}$ of $A$ we have
$$\Hom_A(\rho,\chi'_{\psi'})=(\pi_{S_{T},\varphi,\varphi'+\psi+\psi'})^*=0.$$

\eqref{it:RealStep} Note that for any $X\in \fa$, the action of the Lie algebra generated by $S_T$ and $X$ on $\rho$ can be extended to the action of the corresponding subgroup of $G$. This subgroup is  isomorphic to the group $P_2(\R)$ of affine transformations of the line. Thus we use Proposition \ref{prop:AV0} instead of Lemma \ref{lem:BZ} and continue as in \eqref{it:padicStep}.
\end{proof}


\begin{lem}\label{lem:step2}
Assume that $\pi$ is
 $(S_t,\varphi,\varphi'+\psi)$-distinguished. Then for some (possibly zero) $\psi'\in (\fg^*)^{S_T}_{-1}$, $\pi$ is $(S_T,\varphi+\psi,\varphi'+\psi')$-distinguished.
\end{lem}
\begin{proof}
Consider the form $\omega'(X,Y):=(\varphi+\psi+\varphi')([X,Y])$. The restrictions of this form to $\fr_t$ and to $\fv_T$ are trivial. Thus there exists a maximal totally isotropic subspace $\fl\subset \fu_T$ with $\fr_t\subset \fl$.  Let $L:=\Exp(\fl)$. Since $\fv_T\subset \fr_t$, the characters of $A:=L/R_t$ are given by a quotient of $(\fg^*)^{S_T}_{-1}$. Thus, by Lemma \ref{lem:BZ} and Corollary \ref{cor:ExFun}, for some $\psi'\in (\fg^*)^{S_T}_{-1}$ we have
$$0\neq \Hom_{A}(\pi_{R_{t},\chi},\chi'_{\psi'})=(\pi_{S_T,\varphi+\psi,\varphi'+\psi'})^*.$$
\end{proof}

\subsection{Key propositions}\label{subsec:key}
Let $S,\varphi, h, f,e,Z,S_t$ be as before.
\begin{prop}\label{prop:isom}
Let $T>t \geq 0$. \Dima{Let $\varphi'\in (\fg^*)^{S_t}_{>-2}\cap(\fg^*)^{S_T}_{> -2}$}. Then we have an epimorphism  $$\nu:  \pi_{S_t,\varphi,\varphi'}\onto \pi_{S_T,\varphi,\varphi'}.$$ Moreover, if  $\pi$ is not $(S_s,\varphi,\varphi'+\psi')$-distinguished
 for any $s \in (t,T)$ and any non-zero $\psi' \in (\fg^*)^{S_{s}}_{-1}\cap (\fg^*)^e$ then
\begin{enumerate}[(i)]
\item If $F$ is non-Archimedean then $\nu$ is an isomorphism.
\item $\pi$ is $(S_t,\varphi,\varphi')$-distinguished if and only if $\pi$ is $(S_T,\varphi,\varphi')$-distinguished.
\end{enumerate}
\end{prop}
\begin{proof}
Let $t_0:=t, \, t_1,\dots,t_{n-1}$ be all the critical values between $t$ and $T$ and $t_n:=T$. By Lemmas \ref{lem:lrw}, \ref{lem:Frob2}, and  \ref{lem:Heis}  we have
\begin{equation*}
\pi_{S_t,\varphi,\varphi'} \simeq (\pi\otimes\chi)_{R_t}\onto (\pi\otimes\chi)_{L_{t_1}}\simeq (\pi\otimes\chi)_{R_{t_1}}\onto \dots  \onto (\pi\otimes\chi)_{L_{t_n}}\simeq  \pi_{S_T,\varphi,\varphi'}.
\end{equation*}

The ``moreover" part follows from Lemma \ref{lem:step}.
\end{proof}

\begin{prop}\label{prop:step}
Let $t\geq0$ and  let  $\eta\neq 0 \in (\fg^*)^{S_{t}}_{>-2}\cap\DimaS{(\fg^*)^{S_{t}}_{\leq -1}\cap}(\fg^*)^e$. Suppose that $\pi$ is $(S_t,\varphi,\eta)$-distinguished. Then there exist $T>t, \, \Phi\in (\fg^*)^{S_T}_{-2}$, and $\Phi'\in (\fg^*)^{S_T}_{>-2}$ such that $\varphi \in \overline{G\Phi}\setminus G\Phi$ and $\pi$ is $(S_T,\Phi,\Phi')$-distinguished.
\end{prop}

\begin{proof}

Since $\eta\in (\fg^*)^e\subset \fg^h_{\geq 0},$ we have $\eta\in\fg^Z_{<0}$. Thus for some $s>t$ there exist $\psi\in (\fg^*)^{S_{\Dima{s}}}_{-2}\cap (\fg^*)^e$ and $\varphi'\in (\fg^*)^{S_{\Dima{s}}}_{>-2}\cap  (\fg^*)^e$ such that $\psi\neq 0$ and $\eta=\varphi'+\psi$. \Dima{Note that $\varphi'\in (\fg^*)^{S_{s'}}_{>-2}$  for any $s'\in [t,s]$.}

Let $a_0:=t$, let $a_1,\dots ,a_{m-1}$ be the critical values between $t$ and $s$ and $a_m:=s$. We prove the statement by induction on $m$.

The base case is $m=1$, i.e. there are no critical values between $t$ and $s$. Take $T:=s$. Then Lemma \ref{lem:step2} implies that $\pi_{S_{T},\varphi+\psi,\varphi'+\psi'}\neq 0$ for some $\psi'\in (\fg^*)^{S_{T}}_{-1}.$
Denote $\Phi:=\varphi+\psi$ and $\Phi':=\varphi'+\psi'$.

Note that $\varphi \in \overline{G\Phi}$. Indeed, by Corollary \ref{cor:tor},
there exists an algebraic group morphism  $\nu:\F^{\times}\to \DimaE{G^{ad}}$ and a central element $Z'\in \fg$ such that $Z-Z'\in \Im(d\nu)$. Let $\lambda\in F^{\times}$ be small and $g:=\nu(\lambda)$. Then $Ad^*(g)\varphi=\varphi$ and $Ad^*(g^n)\psi\to 0$. \DimaE{Note also that $G\Phi=G^{ad}\Phi$.}

Note that $\Phi$ belongs to the Slodowy slice to $G\varphi$ at $\varphi$ and thus $\varphi \notin   G\Phi$.

\DimaS{For the induction step, note that by Lemma \ref{lem:step}, $\pi$ is $(S_{a_1},\varphi,\eta+\psi'')$-distinguished for some (possibly zero) $\psi''\in (\fg^*)^{S_{a_1}}_{-1}\cap (\fg^*)^e.$ The proposition follows now from the induction hypothesis.}
\end{proof}
Note that it is possible that $\Phi'=0$.
\Dima{
\begin{example}
Let $G:=\GL_4(F), \, h:=\diag(1,-1,1,-1), Z:=\diag(0,0,1,1), \, t:=3$. Identify $\fg$ with $\fg^*$ using the trace form and let $\varphi:=f:=E_{21}+E_{43}, \, \eta:=E_{14}$, where $E_{ij}$  are elementary matrices. Then $e=E_{12}+E_{34}$ and $\eta\in \fg^{S_t}_{-1}\cap \fg^e$. We have $s=4, \, m=1, \, \varphi'=0, \, \psi =\eta, \, \Phi=\varphi+\psi$.
Then $\Phi$ is regular nilpotent and $\varphi \in \overline{G\Phi} \setminus G\Phi$. Since $\fg_1^{S_{4}}=0$, we have
$\Phi'=0$.
\end{example}
For the next proposition we will need a couple of geometric lemmas, and a definition.
}

\begin{lem}\label{lem:GHconj}
Let
$\psi\in (\fg^*)^S_{-2}\cap (\fg^*)^Z_{>0}$. Assume that $\varphi+\psi\in G\varphi$. Then $\varphi+\psi\in G_S\varphi$.
\end{lem}
\begin{proof}
%

By Corollary \ref{cor:tor}, there exists an algebraic group morphism $\nu:F^{\times}\to \DimaE{G^{ad}}$ and a central element $C\in \fg$ such that $Z-C\in \Im(d\nu)$. Since $Z$ commutes with $\varphi$ and $\psi\in (\fg^*)^Z_{>0}$, this implies that there exists a sequence $t_n\to 0\in F$ with
$\varphi+t_n\psi\in G_S(\varphi+\psi)$ for every $n$.
Thus $\varphi+t_n\psi\in G\varphi$ for every $n$.
Consider  the decomposition
$\fg^*=(\fg^*)^e\oplus ad^*(\fg)(\varphi)$. Since $ad(S)$ preserves all these spaces we have
$$(\fg^*)^S_{-2}=(\fg^*)^e\cap (\fg^*)^S_{-2}+ad^*(\fg^{S}_0)(\varphi),$$  Thus the map
$$\mu:((\fg^*)^e\cap (\fg^*)^S_{-2})\times G_S\to (\fg^*)^S_{-2} \text{ given by }\mu(X,g):= X+Ad^*(g)\varphi$$
is a submersion.
Hence its image contains an open neighborhood of $\varphi$. Thus  $\Im\mu$ contains $\varphi+t_n\psi$ for some $n$. Since $\varphi+t_n\psi\in G\varphi$, and the Slodowy slice $\varphi+(\fg^*)^e$ is strongly transversal to $G\varphi$, we obtain  $\varphi+t_n\psi\in G_S\varphi$ and thus  $\varphi+\psi\in G_S\varphi$.
\end{proof}

\begin{defn}
We will say that $t>1$ is quasi-critical if either $\fg^{S_t}_{1}\nsubseteq \fg^{Z}_{0}$ or $\fg^{S_t}_{2}\nsubseteq \fg^{Z}_{0}$.
We denote by $in(S,\varphi)$ the number of all quasi-critical $t>1$.
\end{defn}
Let us now show that $in(S,\varphi)$ does not depend on the decomposition $S=h+Z$.

\begin{lemma}[{\cite[\S 11]{Bou}}]Let $h'\in \fg_{S}$ be a neutral element for $f$.
\begin{enumerate}[(i)]
\item $\Im(\ad(f))\cap \Ker(\ad(f))$ is a subalgebra in $\fg$, which includes $h-h'$ and lies in $\fg_{<0}^h$.
\item Let $\fn\subset \fg$ be a subalgebra such that all $Y\in \fn$ are nilpotent and $[h,\fn]=\fn$. Then $\exp(\ad(\fn))h=h+\fn$.
\end{enumerate}
\end{lemma}

\begin{lemma}\label{lemma:trivial_quasi-Whittaker_functional}
Let $h'\in \fg_{S}$ be a neutral element for $f$.
Then there exists a nilpotent element $X\in \fg_S$ such that $\exp(\ad(X))(h)=h'$.
\end{lemma}
\begin{proof}
Let $\fb:=\Im(\ad(f))\cap \Ker(\ad(f))\cap \fg_S$. By the previous lemma this is a subalgebra that includes $Y:=h'-h$ and all its elements are nilpotent.  It is easy to see that $[h,\fb]=\fb$, and thus there exists $X\in \fb$ such that $\exp(\ad(X))(h)=h+Y=h'$.
\end{proof}

\begin{cor}\label{en:indep}
The number $in(S,\varphi)$ depends only on $(S,\varphi)$ and not on $h$.
\end{cor}

\begin{prop}\label{prop:quasi2gen}
If  $\varphi\in \QWS(\pi)$ then $\pi_{\varphi}\neq 0$.
\end{prop}

\begin{proof}
Since $\varphi\in \QWS(\pi)$, $\pi$ is $(S',\varphi,\varphi')$-distinguished for some $S',\varphi'$. Without loss of generality we can assume $S=S'$.
Suppose first that $Z=0$. In this case we can assume $\varphi'\in (\fg^*)^h_{-1}.$ Also, in this case the form on $\fu=\fg^h_{\geq 1}$ given by $\omega(X,Y)=\varphi([X,Y])$ has no radical. Thus we can choose a Lagrangian subspace of $\fu_{t}$ on which $\psi'$ vanishes. Thus $\pi_{h,\varphi,\varphi'}\neq0$ implies $\pi_{h,\varphi}\neq0$.

Now we assume $Z\neq 0$ and prove the proposition by induction on $in(S,\varphi)$.
For the base assume that $in(S,\varphi)=0$,  and let $t$ be such that all  the positive eigenvalues of $tZ$ are bigger than all the eigenvalues of $h$ by at least 2. Then we have
$(\fg^*)^Z_{>0}\subset (\fg^*)^{S_{t}}_{\geq2},$ and  $(\fg^*)^Z_{<0}\subset (\fg^*)^{S_{t}}_{\leq-2}.$ This implies
\begin{equation}\label{=BigZ}
(\fg^*)^{S_{t}}_{>-2}= (\fg^*)^{S_{t}}_{>-1}\oplus((\fg^*)^{h}_{-1}\cap (\fg^*)^{Z}_{0}),
\end{equation}
and by Lemma \ref{lem:step2}, $\pi$ is $(S_{t},\varphi,\psi')$-distinguished for some $\psi'\in (\fg^*)^{S_{t}}_{>-2}$.
By \eqref{=BigZ} we have $\psi'=\eta_1+\eta_2$ with $\eta_1\in (\fg^*)^{S_{t}}_{>-1}$ and $\eta_2\in (\fg^*)^{h}_{-1}\cap (\fg^*)^{Z}_{0}$. Then $\eta_1$ vanishes on $\fu_{t}$, and $\eta_2$ vanishes on the radical of the form $\omega_{\varphi}$ on $\fu_{t}$. Thus we can choose a maximal isotropic subspace of $\fu_{t}$ on which $\psi'$ vanishes. Thus $\pi_{S_{t},\varphi,\psi'}=\pi_{S_{t},\varphi}$. By Proposition \ref{prop:isom}, $\pi_{\varphi}$ maps onto $\pi_{S_{t},\varphi}$. Since $\pi_{S_{t},\varphi,\psi'}\neq0$ we obtain $\pi_{\varphi}\neq 0$.

For the induction step  let $t>1$ be the smallest quasi-critical number.  By Lemma \ref{lem:step2},  $\pi$ is $(S_{t},\varphi+\psi,\eta')$-distinguished  for some $\eta'\in (\fg^*)^{S_{t}}_{>-2}$ and some $\psi\in (\fg^*)^{S_{t_{i+1}}}_{-2}\cap (\fg^*)^{Z}_{>0}$. Then $\varphi\in \overline{G(\varphi+\psi)}$. Since $\varphi\in \QWS(\pi)$, we have $\varphi\in G(\varphi+\psi)$, and by Lemma \ref{lem:GHconj} $\varphi= g(\varphi+\psi)$ for some $g\in G_{S_{t}}$.
 Conjugating by $g$ we get $in(S_t,\varphi+\psi)=in(S_t,\varphi)<in(S,\varphi)$. The induction hypothesis implies now that $\pi_{\varphi+\psi}\neq 0$. Thus $\pi_{\varphi}\neq 0$.
\end{proof}

\begin{cor}\label{cor:quasi2gen}
We have $\QWS(\pi)=\WS(\pi)$.
\end{cor}

\begin{example}
Let $G:=\GL_6(F)$  and
$$ h:=\diag(1,-1,1,-1,1,-1),\quad Z:=\diag(0,0,3,3,2.5,2.5).$$ Identify $\fg$ with $\fg^*$ using the trace form and let
$$\varphi:=E_{21}+E_{43}+E_{65}, \quad \varphi':=E_{14}+E_{45}.$$
\DimaO{Let $\pi$ be $(S,\varphi,\varphi')$-distinguished and let us show that $\varphi \notin \QWS(\pi)$.}
Then the first quasicritical  value of $t$ is $t=4/3$. We have $S_{4/3}=\diag(1,-1,5,3,4\frac{1}{3},2\frac{1}{3})$. Then $E_{14}\in \fg^{S_{4/3}}_{-2}$ and $E_{45}\in \fg^{S_{4/3}}_{-4/3}$.
By Lemma \ref{lem:step2}, $\pi$ is $(S_{4/3},\varphi+E_{14},E_{45})$-distinguished, since $\fg^{S_{4/3}}_{\DimaG{-1}}=0$. Now, $\varphi\in \overline{G(\varphi+E_{14})}\setminus G(\varphi+E_{14})$ and thus $\varphi\notin \QWS(\pi)$.
\end{example}

\begin{example}Let $G:=\GL_6(F),S=\diag(1,-1,5,3,4\frac{1}{3},2\frac{1}{3}),\, \varphi=E_{21}+E_{43}+E_{65}+E_{14},$
$$\varphi'=E_{45},\ h=\diag(-1,-3,3,1,1,-1),\ Z=\diag(0,0,0,0,4/3,4/3).
$$
The first quasicritical  value of $t$ is $3/2$.
\DimaO{Now, Proposition \ref{prop:isom} implies that any
$(S,\varphi,\varphi')$-distinguished representation $\pi$ is
$(S_{3/2},E_{21}+E_{43}+E_{65}+E_{14}+E_{45})$-distinguished, and thus has $\pi_{E_{21}+E_{43}+E_{65}+E_{14}+E_{45}}\neq 0$.}
\end{example}

\subsection{Proof of Theorem \ref{thm:MaxModAction}}\label{subsec:PfMaxMod}
First let us show that for all critical $t>0$ and all non-zero $\varphi' \in (\fg^*)^{S_{t}}_{-1}\cap(\fg^*)^e$ we have  $\pi_{S,\varphi,\varphi'}=0$. Suppose the contrary. Then by Proposition \ref{prop:step} for some $T>0$ there exist $\Phi\in (\fg^*)^{S_T}_{-2}$ and $\Phi'\in (\fg^*)^{S_T}_{>-2}$ such that $\varphi \in  \overline{G\Phi}\setminus G\Phi$ and $\pi$ is $({S_T,\Phi,\Phi'})$-distinguished. Thus there exists $\cO\in\QWS(\pi)$ that includes $\Phi$ in its closure. By Corollary \ref{cor:quasi2gen} we have $\cO\in \WS(\pi)$, which contradicts the assumption $G\varphi\in \WS(\pi)$.

Now let $0=t_0< t_1<t_2<\dots<t_{n+1}=1$ be all the critical $ t\in [0,1]$. By Proposition \ref{prop:isom}, we have a sequence of $\Dima{(\tilde G_{\gamma})_Z}$ -equivariant epimorphisms
\begin{equation}\label{=epi}
\pi_{\varphi}= \pi_{S_{t_0},\varphi}\onto \pi_{S_{t_{1}},\varphi}\onto \cdots \onto \pi_{S_{t_n},\varphi}\onto\pi_{S_{t_{n+1}},\varphi}=\pi_{S,\varphi}
\end{equation}
that in the p-adic case are isomorphisms, and in the real case are non-zero. \proofend

\begin{remark}
Under the assumption that $\pi$ is unitary one might be able to construct an invariant scalar product on $\pi_{\varphi}$ and deduce that the epimorphism  of $\pi_{\varphi}$ onto $\pi_{S,\varphi}$ is an isomorphism also for $F=\R$.
\end{remark}

\Dima{
\begin{example}\label{ex:GLSame}
Let $G:=\GL(4,\F)$ and let $S$ be the diagonal matrix $\diag(3,1,-1,-3)$. Identify $\fg$ with $\fg^*$ using the trace form and let
$f:=\varphi:=E_{21}+E_{43}$, where $E_{ij}$ are elementary matrices.
Then we have $S=h+Z$ with $h=\diag(1,-1,1,-1)$ and $Z=\diag(2,2,-2,-2)$.
Thus $S_t=\diag(1+2t,-1+2t,1-2t,-1-2t)$ and the weights of $S_t$ are as follows:
$$\left(
   \begin{array}{cccc}
     0 & 2 & 4t & 4t+2 \\
     -2 & 0 & 4t-2 & 4t \\
     -4t & -4t+2 & 0 & 2 \\
     -4t-2 & -4t & -2 & 0 \\
   \end{array}
 \right).$$
The  critical numbers are $1/4$ and $3/4$. For $t\geq3/4$, the  degenerate Whittaker model $\cW_{S_t,\varphi}$ is the induction $\ind_{N}^G\chi_{\varphi}$, where $N$ is the group of upper-unitriangular matrices.
The sequence of inclusions $\fr_{0}\subset \fl_{1/4}\sim \fr_{1/4}\subset \fl_{3/4}=\fr_{3/4}$  is:

\begin{equation}\label{=ThmEx}
\left(
   \begin{array}{cccc}
     0 & - & 0 & - \\
     0 & 0 & 0 & 0 \\
     0 & - & 0 & - \\
     0 & 0 & 0 & 0 \\
   \end{array}
 \right) \subset \left(
   \begin{array}{cccc}
     0 & - & a & - \\
     0 & 0 & 0 & a \\
     0 & * & 0 & - \\
     0 & 0 & 0 & 0 \\
   \end{array}
 \right)\sim
\left(
   \begin{array}{cccc}
     0 & - & * & - \\
     0 & 0 & 0 & * \\
     0 & 0 & 0 & - \\
     0 & 0 & 0 & 0 \\
   \end{array} \right)
\subset
\left(
   \begin{array}{cccc}
     0 & - & - & - \\
     0 & 0 & * & - \\
     0 & 0 & 0 & - \\
     0 & 0 & 0 & 0 \\
    \end{array} \right)
\end{equation}
   Here, both $*$ and $-$ denote arbitrary elements. $-$ denotes the entries in $\fv_t$ and $*$ those in $\fw_t=\fg_{1}^{S_t}$. The letter $a$ denotes an arbitrary element, but the two appearances of $a$ denote the same numbers.
   The  passage from $\fl_{1/4}$ to $\fr_{1/4}$ is denoted by $\sim$. At  $3/4$ we have $\fl_{3/4}=\fr_{3/4}.$

Let $\pi\in \Rep^{\infty}(G)$ with $G\varphi\in \WS(\pi)$.
The sequence of epimorphisms \eqref{=epi} is given by the sequence of inclusions \eqref{=ThmEx}. To see that these epimorphisms are non-zero (and are isomorphisms for $F\neq \R$) we need to analyze the dual spaces to $\fw_{1/4}^f$ and $\fw_{3/4}^f$.
These spaces are spanned by $E_{13}+E_{24}$ and by $E_{23}$ respectively. Thus, the dual spaces are spanned   by $E_{31}+E_{42}$ and by $E_{32}$ respectively.
Note that the joint centralizer of $h,Z$ and $\varphi$ in $G$ acts on these spaces by scalar multiplications, identifying all non-trivial elements.
By Proposition \ref{prop:isom} it is enough to show that $\pi_{S_{1/4},\varphi,E_{31}+E_{42}}=0$ and $\pi_{S_{3/4},\varphi,E_{32}}=0$.

This is guaranteed by Propositions \ref{prop:step} and \ref{prop:quasi2gen}, but for the sake of the example let us show this more directly.

First assume by way of contradiction that $\pi_{S_{3/4},\varphi,E_{32}}\neq 0$. Note that $E_{32}\in \fg^{S_{1}}_{-2}$ and that $\fw_1=0 $. Thus $\fu_1=\fl_1=\fr_{3/4}$ and $$\pi_{S_{1},\varphi+E_{32}}\simeq \pi_{S_{3/4},\varphi,E_{32}}\neq 0.$$ Note that $\Phi:=\varphi+E_{32}=E_{21}+E_{43}+E_{32}$ is a regular nilpotent element, and $S_{1}=S=\diag(3,1,-1,-3)$ is a neutral element for it. Thus $\pi_{\Phi}\neq 0,$ contradicting the assumption that $\Phi$ is maximal in $\WS(\pi)$.

Now assume by way of contradiction that $\pi_{S_{1/4},\varphi,E_{31}+E_{42}}\neq 0$. Note that $E_{31}+E_{42}\in \fg^{S_{1/2}}_{-2}$ and that $\fw_{1/2}=0 $. Thus $\fl_{1/2}=\fu_{1/2}=\fr_{1/4}$ and  $$\pi_{S_{1/2},\varphi+E_{31}+E_{42}}\simeq \pi_{S_{1/4},\varphi,E_{31}+E_{42}}\neq 0.$$ Note that $\Psi:=\varphi+E_{31}+E_{42}=E_{21}+E_{43}+E_{31}+E_{42}$ is a regular nilpotent element, and $S_{1/2}=\diag(2,0,0,-2)$ is a neutral element for it. Thus $\pi_{\Psi}\neq 0,$ contradicting the assumption that $G \cdot \varphi$ is maximal in $\WO(\pi)$.
\end{example}
}

\section{Proof of Theorems  \ref{thm:cusp} and \ref{thm:MaxFin}}\label{sec:cors}
\subsection{Proof of Theorem \ref{thm:MaxFin}}\label{subsec:MaxFin}

\DimaP{In the non-Archimedean case, it is enough to prove that for any \DimaD{homomorphism} $\nu:\widetilde{SL_2(F)}\into \widetilde{M_{\gamma}}$, the image acts on $\pi_{\varphi}$ by $\pm \Id$.
In the Archimedean case, \DimaQ{by definition of $\widetilde{M_{\gamma}}$ its Lie algebra $\fm$ is generated by nilpotent elements. Thus, by Lemmas \ref{lem:Finsl2} and \ref{lem:FinPBW}, it is enough to prove that any nilpotent element of $\fm$ acts locally finitely on $\pi^*_{\varphi}$. Since any such nilpotent lies in the image of the \DimaD{differential of a homomorphism} of the form $\nu:\widetilde{SL_2(F)}\into \widetilde{M_{\gamma}}$, it suffices to show that the restriction of $\pi^*_{\varphi}$ to the image of any $\nu$ as above is locally finite.}
By Lemma \ref{lem:BZ} and Corollary \ref{cor:FinQuot}, in both cases} it is enough to show that the restriction \DimaS{of $\pi_{\varphi}$ to the image of $\nu$} is non-generic.

  Fix such a \DimaQ{morphism} $\nu$ and let $(e',h',f')$ be the corresponding $\sl_2$-triple in $\fg_{\gamma}$, and let $\varphi'\in \fg^*$ denote the nilpotent element given by the Killing form pairing with $f'$.
\DimaS{Let $S_t:=h+th'$.
Let $0=t_0< t_1<t_2<\dots<t_{n+1}=2/3$ be all the critical $ t\in [0,2/3]$. By Corollary \ref{cor:quasi2gen}, $\WS(\pi)=\QWS(\pi)$. Thus $\varphi\in \QWS(\pi)$, and Proposition \ref{prop:step} implies that
\begin{equation}\label{=q0}
\pi_{S_t,\varphi,\psi}=0 \text{ for any }t \in (0,1/2) \text{ and any non-zero } \psi' \in (\fg^*)^{S_{t}}_{-1}\cap (\fg^*)^e
\end{equation}

 By Proposition \ref{prop:isom}, this implies that we have a sequence of epimorphisms
\begin{equation}\label{=epi2}
\pi_{\varphi}= \pi_{S_{t_0},\varphi}\onto \pi_{S_{t_{1}},\varphi}\onto \cdots \onto \pi_{S_{t_n},\varphi}\onto\pi_{S_{t_{n+1}},\varphi}=\pi_{S,\varphi}
\end{equation}
that in the p-adic case are isomorphisms, and in the real case are non-zero.
By Lemma \ref{lem:actions}, these epimorphisms commute with the action of $\Exp(e')$. Let $\chi'$ denote the character of  $\Exp(e')$ given by $\varphi'$, and denote
$(\pi_{S_{t_i},\varphi})_{e',\varphi'}:=(\pi_{S_{t_i},\varphi})_{\Exp(e'),\chi'}$. Then  \eqref{=epi2} induces a sequence of epimorphisms
\begin{equation}\label{=epi3}
(\pi_{\varphi})_{\Exp(e'),\chi'}= (\pi_{S_{t_0},\varphi})_{e',\varphi'}\onto (\pi_{S_{t_{1}},\varphi})_{e',\varphi'}\onto \cdots \onto (\pi_{S_{t_n},\varphi})_{e',\varphi'}\onto(\pi_{S_{t_{n+1}},\varphi})_{e',\varphi'}
\end{equation}
We note that the last element of the sequence is zero, since $e'\in \fv_{t_{n+1}}$ and $\varphi(e')=0$,
 and thus $\Exp(e')$ acts trivially on $\pi_{S_{t_{n+1}},\varphi}$.
In order to show that the restriction of  $\pi_{\varphi}$ to the image of $\nu$ is non-generic it is enough to show that all the spaces in  \eqref{=epi3} vanish.

 If $F$ is $p$-adic then this is straightforward, since in this case all the maps in \eqref{=epi2} are isomorphisms, and thus so are the maps in \eqref{=epi3}.
Let us show that in the Archimedean case $(\pi_{S_{t_i},\varphi})_{e',\varphi'}=0$, by backwards induction on $i$. The base case is $i=n+1$. For the induction step, recall that the map $(\pi_{S_{t_{i-1}},\varphi})\onto(\pi_{S_{t_{i}},\varphi})$ is given by
\begin{equation*}
\pi_{S_{t_{i-1}},\varphi}\cong \pi_{R_{t_{i-1}},\varphi} \onto \pi_{L_{t_{i}},\varphi}\cong \pi_{S_{t_{i-1}},\varphi}.
\end{equation*}
Consider the dual map $(\pi_{L_{t_{i}},\varphi})^*\into (\pi_{R_{t_{i-1}},\varphi})^*$. Its image is the space of invariants under the commutative Lie algebra $\fa:=\fl_{t_{i}}/\fr_{t_{i-1}}$. By Lemma \ref{lem:step} and \eqref{=q0}, $\fa$ acts locally nilpotently on $(\pi_{R_{t_{i-1}},\varphi})^*$. Since $[e',\fl_{t_{i}}]\subset \fr_{t_{i-1}},$ the actions of $e'$ and of $\fa$ commute. Thus $\fa$ preserves the space of
$(e',\varphi')$-semi-invariants
$((\pi_{R_{t_{i-1}},\varphi})^*)^{e',\varphi'}$, and acts on it locally nilpotently. But $\fa$ has no invariants on this space by the induction hypothesis, since this space of invariants is dual to $(\pi_{S_{t_i},\varphi})_{e',\varphi'}=0$. Thus $((\pi_{R_{t_{i-1}},\varphi})^*)^{e',\varphi'}=0$, and thus $(\pi_{S_{t_{i-1}},\varphi})_{e',\varphi'}=0$.
\proofend

\begin{remark}
Let  $(S,\varphi)$ be a  Whittaker pair with $\varphi\in \WS(\pi)$, and let $G(S,\varphi)\subset G$ be the subgroup defined in Definition \ref{def:act}.
The same argument shows that the cover of the subgroup of $G(S,\varphi)$ generated by unipotent elements acts locally finitely on $\pi_{S,\varphi}^*$ if $F$ is Archimedean, and acts on $\pi_{\varphi}$ by $\pm 1$ if $F$ is  non-Archimedean.
\end{remark}
}

\subsection{Proof of Theorem \ref{thm:cusp}}\label{subsec:Cusp}
\DimaA{By a quasi-cuspidal $\pi$ we mean a smooth (not necessarily admissible or finitely-generated) representation such that the Jacquet module $r_P(\pi)$ vanishes for any proper parabolic subgroup $P\subset G$.}

Let $\pi$ be quasicuspidal and let $\cO\in \WS(\pi)$. Suppose by way of contradiction that $\cO$ is not $F$-distinguished. Thus there exists a proper parabolic subgroup $P\subset G$, a
Levi subgroup $L\subset P$ and a nilpotent $f\in \fl$ such that $\varphi\in \cO$, where $\varphi\in \fg^*$ is given by the Killing form pairing with $f$.  Let $h$ be a neutral element for $f$ in $\fl$.
Choose a rational-semisimple element $Z\in \fg$ such that $L$ is the centralizer of $Z$, $\fp:=\fg^{Z}_{\geq 0}$ is the Lie algebra of $P$, and all the positive eigenvalues of $Z$ are bigger than all the eigenvalues of $h$ by at least 2. Note that $\fn:=\fg^{Z}_{> 0}$ is the nilradical of $\fp$. Let $S:=h+Z$. By construction we have $\fn\subset \fg^S_{>2}$ and thus the degenerate Whittaker quotient $\pi_{S,\varphi}$ is a quotient of $r_P\pi$. By Theorem \ref{thm:MaxModAction}, the maximality of $\cO$ implies $\pi_{S,\varphi}\simeq \pi_{\varphi}$.
Thus $r_P\pi$ does not vanish, in contradiction with the condition that $\pi$ is quasi-cuspidal. \proofend

\section{Proof of Theorem \ref{thm:adm}, and relation to admissible and special orbits}\label{sec:PfAdm}

\subsection{Proof of Theorem \ref{thm:adm}}\label{subsec:PfA}
Fix a nilpotent $\varphi \in \fg^*$.
\DimaC{
Let $\gamma=(e,h,f)$ be an $\sl_{2}$-triple such that $\varphi$ is given by pairing with $f$ under the Cartan-Killing form on $\g$.

Let $(G_{\gamma})_{ss}$ be the  the  subgroup of $G_{\gamma}$ generated by the exponents of the derived algebra for $\g_{\gamma}$, and $\widetilde{(G_{\gamma})_{ss}}$ be the corresponding subgroup of $\widetilde{G_{\gamma}}$.

Let $\widetilde{K}_{\gamma}\subset \widetilde{(G_{\gamma})_{ss}}$ be the anisotropic (and hence compact) part.

\begin{prop}\label{prop:CommTilde}
If $F=\R$ then
\begin{equation}\label{=CommTilde}
\widetilde{(G_{\gamma})_{ss}}
\cong \widetilde{M}_{\gamma}\times \widetilde{K}_{\gamma}/\Delta(\widetilde{M}_{\gamma}\cap \widetilde{K}_{\gamma}),
\end{equation}
\end{prop}
\begin{proof}
For the Lie algebras we have $[\fg_{\gamma},\fg_{\gamma}]=\fk_{\gamma}\oplus \fm_{\gamma}$, thus
\begin{equation}\label{=Comm}
({G}_{\gamma})_{ss}\cong {M}_{\gamma}\times {K}_{\gamma}/\Delta({M}_{\gamma}\cap {K}_{\gamma}).
\end{equation}
Thus, in order to prove \eqref{=CommTilde} it is enough to show that $\widetilde{M_{\gamma}}$ and $\widetilde{K_{\gamma}}$ commute. Fix $\tilde k\in \widetilde{K_{\gamma}}$. By \eqref{=Comm}, the commutator map $\tilde m \mapsto \tilde m \tilde k \tilde m^{-1}\tilde k^{-1}$ maps $\widetilde{M_{\gamma}}$ to $\{1,\eps\}$. Suppose, by way of contradiction, that the image is non-trivial. Then $\widetilde{M_{\gamma}}$ is disconnected and thus $\widetilde{M_{\gamma}}\simeq \Z_2 \times M_{\gamma}$. Thus for some $m\in M_{\gamma}\subset \widetilde{M_{\gamma}}$, the commutator map $\tilde k \mapsto  m \tilde k  m^{-1}\tilde k^{-1}$ is non-trivial. Thus $\widetilde{K_{\gamma}}$ also splits, which implies that it commutes with $\widetilde{M_{\gamma}}$.
\end{proof}
}
\DimaC{

\begin{prop}\label{prop:MtoGss}
If $\widetilde{M_{\gamma}}$ has a genuine finite-dimensional representation then so does $\widetilde{(G_{\gamma})_{ss}}$.
\end{prop}
\begin{proof}
Let $\rho_0$ be a genuine finite-dimensional representation of $\widetilde{M_{\gamma}}$.

Assume first $F\neq \R$. Let $C'$ denote the kernel of $\rho_0$ and $C$ denote the projection of $C'$ to $M_{\gamma}$. Then $C$ is a normal open subgroup of $M_{\gamma}$ and thus $C=M_{\gamma}$. Since $\rho$ is genuine, the projection $C' \to C$ is an isomorphism and thus defines a splitting of $\widetilde{M_{\gamma}}$.
As in the proof of Proposition \ref{prop:CommTilde} one shows that there exists an open subgroup $K'\subset \widetilde{K_{\gamma}}$ that commutes with $C'$ and includes $\eps$.
Since the quotient $(K'C')/C'\cong K'/(K'\cap C')$ is compact, it has a finite-dimensional genuine representation $\tau$. By composing with the natural projection, $\tau$ lifts to a representation of $K'C'$. Since $K'C'$ is of finite index in $\widetilde{(G_{\gamma})_{ss}}$, the induction of $\tau$ to $\widetilde{(G_{\gamma})_{ss}}$ is still finite-dimensional (and genuine).

Assume now that $F=\R$.
Then $$(\ind^{\widetilde{(G_{\gamma})_{ss}}}_{\widetilde{M}_{\gamma}}\rho_{0})|_{\widetilde{K}_{\gamma}} =\ind^{\widetilde{K}_{\gamma}}_{\widetilde{K}_\cap \widetilde{M}_{\gamma}}(\rho_{0}|_{\widetilde{K}_\cap \widetilde{M}_{\gamma}}).$$
By the Peter-Weyl theorem, this implies that the induction has a finite-dimensional $\widetilde{K_{\gamma}}$-isotypic component $\rho$. By Proposition \ref{prop:CommTilde}, $\widetilde{M_{\gamma}}$ preserves $\rho$ and thus $\rho$ is a genuine finite-dimensional representation of  $\widetilde{(G_{\gamma})_{ss}}$.
\end{proof}

\begin{prop}\label{prop:MtoG}
If $\widetilde{M_{\gamma}}$ has a genuine finite-dimensional representation then so does $\widetilde{G}_{\varphi}$.
\end{prop}
\begin{proof}
Let $\tilde Z$ denote the center of $\widetilde{G_{\gamma}}$ and let $\tilde H:=\tilde Z\widetilde{(G_{\gamma})_{ss}}$.
 Let us  show that $\widetilde{H}$ has a finite-dimensional genuine representation.
By Proposition \ref{prop:MtoGss}, $\widetilde{(G_{\gamma})_{ss}}$  has an irreducible genuine finite-dimensional representation $(\rho_1,V_1)$.

Notice that $\tilde Z \cap \widetilde{(G_{\gamma})_{ss}} $ acts on $(\rho_1,V_1)$ by a character, that we will denote by $\chi_1$.  By the classical theory of Pontryagin duality for locally compact abelian groups, we can extend the character $\chi_{1}$  to a character of $\tilde Z$, see \cite[Theorem 5]{Dix}. This defines a genuine action of $\tilde H$ on $\rho_1$. Let $\rho_2$ be the induction of this representation to $\widetilde{G_{\gamma}}$. Since $\tilde H$ has finite index in $\widetilde{G_{\gamma}}$, $\rho_2$ is finite-dimensional.
By composing $\rho_2$ with the epimorphism  $\widetilde{G}_{\varphi}\to  \widetilde{G_{\gamma}}$, we obtain a genuine finite-dimensional representation of  $\widetilde{G_{\varphi}}$.
\end{proof}

Theorem \ref{thm:adm} follows now from Theorem \ref{thm:MaxFin} and Proposition \ref{prop:MtoG}.
}


\subsection{Admissible, quasi-admissible and special orbits}\label{subsubsec:RelSpecOrb}

\begin{prop}\label{prop:NAAdm}
Assume that $F$ is non-Archimedean.
Then $\varphi$ is quasi-admissible if and only if the cover $\widetilde{(G_{\varphi})_0}$ (see \S \ref{subsec:cov}) splits over an open normal subgroup of finite index.
\end{prop}
\begin{proof}
First, if the cover $\widetilde{(G_{\varphi})_0}$ (see \S \ref{subsec:cov}) splits over an open normal subgroup $H \subset (G_{\varphi})_0$ of finite index then the cover $\widetilde{H}$ has a one-dimensional genuine representation. The induction of this representation to $\widetilde{(G_{\varphi})_0}$ is still genuine and finite-dimensional.

Now assume that $\widetilde{G_{\varphi}}$ has a genuine finite-dimensional representation $\rho$. Restrict $\rho$ to $\widetilde{(G_{\varphi})_0}$ and let $C$ denote the kernel of the restriction. Then $C$ is an open normal subgroup. Let us show that it has finite index. Indeed, since $\rho$ is finite-dimensional, Lemma \ref{lem:BZ} implies that $C$ includes all the unipotent elements of $\widetilde{(G_{\varphi})_0}$. Thus, $C$ is cocompact and open and hence has finite index. Since $\rho$ is genuine, the restriction of the covering map to $C$ is one-to-one. Thus, the cover splits over the image of $C$ in $(G_{\varphi})_0$.
\end{proof}

\begin{prop}\label{prop:AdmQuasi}
All admissible orbits are quasi-admissible.
\end{prop}
\begin{proof}

Let $\cO$ be an admissible orbit,  let $\varphi \in \cO$ and let $\gamma$ be a corresponding $\sl_2$-triple. Then, by definition, the cover  splits over the group $(G_{\varphi})_0$
generated by exponents of $\fg_{\varphi}$. This group includes $M_{\gamma}$ and thus $\widetilde{M_{\gamma}}$ splits, and has a genuine character $\chi$. By Proposition \ref{prop:MtoG} this implies that $\cO$ is quasi-admissible.
\end{proof}

\begin{remark}
Any $F$-distinguished orbit is quasi-admissible, since for such orbits $M_{\gamma}$ is trivial. Over non-Archimedean $F$, $F$-distinguished orbits are admissible since the metaplectic cover splits over compact subgroups, see \cite[Theorem 4.6.1]{MT}. Over $F=\R$, the minimal orbit in $U(2,1)$ is $\R$-distinguished but not admissible. In general, the $\R$-distinguished orbits for semi-simple groups are classified in \cite[Theorems 8-14]{PT} (under the name compact orbits), and comparing this classification with the classification of admissible orbits \DimaN{given in \cite[Theorem 3]{Oht} for classical groups, and \cite{Noel,Noel2} for exceptional groups,} we see that for the
\DimaC{groups
\begin{equation}\label{=CompProb}
SU(p,q) \text{(with }p,q\geq 1), EII,EV,EVI,EVIII,EIX
\end{equation}
 there exist $\R$-distinguished non-admissible orbits\footnote{For $SU(m,n)$, the $\R$-distinguished orbits are the ones described by partitions in which all rows of the same size have also the same signs, while admissible orbits are described in Theorem \ref{thm:Oht}\eqref{it:SU} below.}. On the other hand, for other real simple groups, all $\R$-distinguished orbits are admissible.
Thus it is possible that for simple groups not appearing in the list \eqref{=CompProb} admissibility is equivalent to quasi-admissibility.
We conjecture that quasi-admissibility is equivalent to the splitting of $\widetilde{M_{\gamma}}$ for all groups.}
\end{remark}

Let us now discuss the relation to special orbits.

\begin{thm}[{\cite[Corollaries 5.9 and 6.3]{Nev}, \cite[Main Theorem]{Nev2}}]\label{thm:Nev}
Let $F$ be non-Archimedean. If $G$ is classical then the set of admissible orbits coincides with the set of special orbits. If $G$ is split exceptional \DimaN{different from $E_8$} then the set of admissible orbits includes the set of special orbits.
\end{thm}
\DimaN{It is conjectured in \cite{Nev2} that the same holds for $E_8$.}

For $F=\R$, the sets of special and admissible orbits coincide for orthogonal, symplectic and general linear groups. However, for unitary groups all orbits are special but most orbits are not admissible. See Theorem \ref{thm:Oht} below for these facts. Also, for several exceptional groups, some split and some non-split, there are special non-admissible orbits and admissible non-special orbits - see \cite{Noel,Noel2}.

It is conjectured that in the non-Archimedean case the Whittaker support consists of special orbits. By \cite{Mog,JLS} this holds  for classical ($p$-adic) groups.

The analogous conjecture cannot hold for exceptional $G$ if $F=\R$.
Namely, for the minimal representation $\pi_{\min}$ of $G_2(\R)$ constructed in \cite{VogG2}, $\WS(\pi)$ consists of the minimal orbit $\cO_{\min}$ of $G_2$, which is admissible but not special.
\DimaN{
Let us explain the notion of minimal representation and the relation to Whittaker support.
\begin{defn}
We call a smooth representation $\pi$ of a real reductive group \emph{minimal}
if its annihilator variety is the closure of the minimal orbit in $\fg^*(\C)$.
The annihilator variety is defined to be the set of common zeros of the symbols of the elements of the annihilator ideal of $\pi$ in the universal enveloping algebra of $\fg$.
\end{defn}
\begin{prop}\label{prop:min}
The Whittaker supports of minimal representations consist of minimal orbits. \end{prop}
\begin{proof}
Let $\pi$ be a minimal representation. By \cite[Corollary 4]{Mat} this implies that $\pi_{\cO}=0$ unless $\cO=\{0\}$ or $\cO$ is minimal. Since $\pi$ is infinite-dimensional, Theorem \ref{thm:MaxFin} implies $\{0\}\notin \WS(\pi_{\min})$.
\end{proof}}

The conjecture \DimaN{on speciality of $\WS(\pi)$} also cannot be extended to complex reductive groups.

Let us now prove Proposition \ref{prop:ClasQuas} that states that admissibility, quasi-admissibility and speciality are equivalent for the groups $O(p,q)$, $\SO(p,q)$ and $\Sp_{2n}(\R)$. Our proof is based on Theorem \ref{thm:Nev} and the following theorem from \cite{Oht}. For the formulation, recall that the nilpotent orbits in real classical groups are given by signed partitions satisfying certain conditions. Fortunately, the signs have no effect on the admissibility and speciality.

\begin{thm}[{\cite[Theorem 3]{Oht}}]\label{thm:Oht}  Let $\cO\subset \fg^*$ be a nilpotent orbit and $\lambda$ be the corresponding partition.
\begin{enumerate}[(i)]
\item \label{it:OSpU} Let $G$ is one of the groups $O(p,q)$, $\SO(p,q)$, $U(p,q)$ or $\Sp_{2n}(\R)$.
Then $\cO$ is admissible if and only if
for each even row (i.e., row with even length) in $\lam$, the number
of odd rows in $\lam$, which are shorter than the even row, is even and for each
odd row in $\lam$, the number of even rows in $\lam$, which are longer than the odd
row, is even.

\item \label{it:SU} Let $G=SU(p,q)$.
Then $\cO$ is admissible if and only if
for each even row (i.e., row with even length) in $\lam$, the number
of odd rows in $\lam$, which are shorter than the even row, is even and for each
odd row in $\lam$, the number of even rows in $\lam$, which are longer than the odd
row, is even.

\item \label{it:Other} For all other real classical groups, all nilpotent orbits are admissible.
\end{enumerate}
\end{thm}

\begin{proof}[Proof of Proposition \ref{prop:ClasQuas}]
First of all, comparing Theorem \ref{thm:Oht}\eqref{it:OSpU} to the description of special orbits in \cite[\S 6.3]{CM}, and using Theorem \ref{thm:Nev} (for the case $F\neq \R$) we see that for the groups $O(V)$, $\SO(V)$, and $\Sp_{2n}(F)$ the set of admissible orbits coincides with the set of special orbits. Next, by Proposition \ref{prop:AdmQuasi}, this set is included in the set of quasi-admissible orbits. It is left to show that non-admissible orbits are not  quasi-admissible either.

We will do it for the symplectic group, since the construction for the orthogonal case is very similar.  Let $\cO$ be a non-admissible orbit and $\lambda$ be the corresponding partition. Since every odd part in $\lambda$ appears with even multiplicity, Theorems \ref{thm:Oht} and \ref{thm:Nev} imply that there exists an odd part $\lambda_i$ in $\lambda$ such that the number of even parts bigger than $\lambda_i$ (counted with multiplicity) is odd.
Let 2m be the multiplicity of $\lambda_i$ in $\lambda$. By  \cite[\S 5.3]{Nev}, the centralizer $G_{\varphi}$ includes a  group $H$ isomorphic to $\Sp_{2m}(F)$, over which the cover does not split. Since $\Sp_{2m}(\R)$ is simple and has no non-trivial algebraic covers, if $F=\R$ then $\widetilde{H}$ cannot have genuine finite-dimensional representations, and thus $\cO$ is not quasi-admissible.
If $F\neq \R$ then $\cO$ is not quasi-admissible by Proposition \ref{prop:NAAdm}.
\end{proof}

\DimaF{
To complete the picture for real classical groups we will need the following lemma.

\begin{lem}[{\cite[\S 3, Lemma 7]{Kaz}}]\label{lem:Kaz}
Let $V$ be a Hermitian space (of arbitrary signature) and $SU(V)$ be the corresponding special unitary group. Consider $V$ as a real vector space and define a symplectic form on $V$ to be the real part of the hermitian form. Then the corresponding metaplectic cover of $SU(V)$ splits.
\end{lem}

\begin{cor}\label{cor:U}
All nilpotent orbits in $SU(p,q)$ and in $U(p,q)$ are quasi-admissible.
\end{cor}

\begin{proof}
By Proposition \ref{prop:MtoG}, it is enough to show that $\widetilde{M_{\gamma}}$ splits for any $\sl_2$-triple $\gamma=(e,h,f)$ in $\fs \fu(p,q)$.
Let $W=\C^{p+q}$ denote the standard representation of $\fg$ and $\langle \cdot, \cdot \rangle$ denote the fixed hermitian form on $W$ of signature $(p,q)$.
Note that $\gamma$ defines a decomposition  $W=\bigoplus_{r\geq 0}W(r),$ where $W(r)$ is the direct sum of all simple $\gamma$-submodules of highest weight $r$.
For each $r$, let $H(r)$ denote the highest weight subspace of $W(r)$, and define a  sesquilinear form on $H(r)$ by $\langle v,w \rangle_r:=\langle v, f^r w\rangle$ if $r$ is even and $\langle v,w \rangle_r:=i\langle v, f^r w\rangle$ if $r$ is odd. Since $f$ is skew-hermitian, the form $\langle \cdot, \cdot \rangle_r$ is hermitian. Note that $\langle v,w \rangle_r$ is non-degenerate for all $r$ and that   $M_{\gamma}$  is isomorphic to $\prod_r SU(H(r))$.

By \cite[\S 5.3]{Nev} (which is written uniformly for all fields $F$), the splitting on $\widetilde{M_\gamma}$ is implied by the splitting of the metaplectic cover of $SU(H(r))\times SU(H(r'))$  inside $\Sp(H(r)\otimes \overline{H(r')})$ for all pairs $(r,r')$ of different parity. However,  $SU(H(r))\times SU(H(r'))$ is a subgroup of $SU(H(r)\otimes \overline{H(r')})$ and by Lemma \ref{lem:Kaz} the metaplectic cover splits on the latter group.
\end{proof}

}

Theorems \ref{thm:Nev} and \ref{thm:Oht} and Corollary \ref{cor:U} imply the following corollary.
\begin{cor}
For classical groups, all special orbits are quasi-admissible.
\end{cor}

It is possible that all special orbits are quasi-admissible for all groups.

\section{Generalized Whittaker models for non-maximal orbits}\label{sec:compar}
\setcounter{lemma}{0}

The notion of quasi-Whittaker model and the method of \S \ref{sec:MaxOrb} allow us to relate  degenerate Whittaker models corresponding to different nilpotent orbits.
Let $(h,\varphi)$ be a neutral pair,  let a rational semi-simple $Z\in \fg$ commute with $h$ and with $\varphi$ and let $S:=h+Z$.
\begin{prop}\label{prop:LowerOrb}
 Let $\psi\in (\fg^*)^h_{\geq -1}\cap(\fg^*)^S_{-2}$. Then we have an epimorphism $$\cW_{h,\varphi,\psi}\onto \cW_{S,\varphi+\psi}.$$
\end{prop}

\begin{proof}
The proof is similar to that of Proposition \ref{prop:isom}.
Let $S_t:=h+tZ$ and let $t_0:=0, \, t_1,\dots,t_{n-1}<1$ be all the critical values between $0$ and $1$ and $t_n:=1$. By Lemmas \ref{lem:lrw}, \ref{lem:Frob2}, 
we have
\DimaH{\begin{equation*}
\cW_{h,\varphi,\psi} \simeq ind^G_{R_0}\chi_{\varphi+\psi}\onto ind^G_{L_{t_1}}\chi_{\varphi+\psi}\simeq ind^G_{R_{t_1}}\chi_{\varphi+\psi}\onto \dots  \onto ind^G_{L_{t_n}}\chi_{\varphi+\psi}\simeq  \cW_{S,\varphi+\psi}.
\end{equation*}}
\end{proof}

This proposition is strengthened by the following lemma.
\begin{lem}\label{lem:quasiIso}
Let $\psi\in (\fg^*)^h_{\geq -1}$. Then we have a natural isomorphism $\cW_{h,\varphi,\psi}\cong \cW_{h,\varphi}$.
\end{lem}
\begin{proof}
\DimaH{Since the form $\omega_{\varphi}$ on $\fg^h_1$ is non-degenerate, we can choose a Lagrangian subspace $\fl'\subset \fg^h_1$ on which $\psi$ vanishes.
Let $\fl:=\fl'\oplus \fg^h_{>1}$. Then $\fl$ is a maximal coisotropic subspace of $\fg_{h\geq 1}$ and $\psi$ vanishes on $\fl$. Let $L:=\Exp(\fl)$. Then
$\cW_{h,\varphi,\psi}\cong \ind_{L}^G\chi_{\varphi+\psi} = \ind_{L}^G\chi_{\varphi}\cong \cW_{h,\varphi}.$}
\end{proof}

\DimaH{Note that $ (\fg^*)^Z_{<0}\cap(\fg^*)^S_{-2}=(\fg^*)^h_{\geq -1}\cap(\fg^*)^S_{-2}$. Thus Proposition \ref{prop:LowerOrb} and Lemma \ref{lem:quasiIso} imply the following corollary.

\begin{cor}
Let $\psi\in (\fg^*)^Z_{<0}\cap(\fg^*)^S_{-2}$. Then we have an epimorphism $\cW_{h,\varphi}\onto \cW_{S,\varphi+\psi}$.
\end{cor}
}
Together with Theorem \ref{thm:MaxModAction} we obtain
\begin{thm}\label{thm:ComparOrbit}
Let $\pi\in \Rep^{\infty}(G)$ and let $\cO\in \WS(\pi)$. Let $(h,\varphi)$ be a neutral Whittaker pair. Suppose that there exists a Whittaker pair $(S,\Phi)$ such that $\Phi\in \cO$, $\varphi\in (\fg^*)^S_{-2}$, $[h,S]=0$, and $\Phi-\varphi\in (\fg^*)^{S-h}_{<0}$. Then $G \cdot \varphi \in \WO(\pi).$
\end{thm}

This theorem is strongest for the group $\GL_n$. In order to apply it to this case we will need the following proposition from linear algebra, that we will prove in the next subsection, following \cite[\S 4.2]{GGS}.

\begin{prop}\label{prop:GL}
Let $\fg_n:=\gl_n(\Q)$ and let $\cO', \cO \subset \fg^{*}_{n}$ be rational nilpotent orbits, with $\cO'\subset \overline{\cO}$. Then for any neutral pair $(h,\varphi)$ with $\varphi\in \cO'$  there exist a rational semi-simple $Z\in \DimaL{(\g_n)^h_0\cap (\g_n)_{\varphi}}$ and $\psi\in  (\fg_n^*)^Z_{<0}\cap (\fg_n^*)^{h+Z}_{-2}$ such that $\varphi+\psi\in \cO$.
\end{prop}

\begin{corollary}\label{cor:GL}
Let $G$ be either $\GL_n(F)$ or $\GL_n(\C)$. Let $\pi\in \Rep^{\infty}(G)$. Let $\cO\in \WO(\pi)$ and $\cO'\subset \overline{\cO}$. Then $\cO'\in \WO(\pi)$.
\end{corollary}
\begin{proof}
Since $\cO\in \WO(\pi)$, there exists $\cO_1\in \WF(\pi)$ with $\cO\subset \overline{\cO_1}$. Then $\cO\subset \overline{\cO_1}$ as well. Choose a neutral pair $(h,\varphi)$ with $\varphi\in \cO'$ and apply Proposition \ref{prop:GL} to this pair and the orbit $\cO_1$. Then set $S:=h+Z$ and $\Phi:=\varphi+\psi$. By Theorem \ref{thm:ComparOrbit} we obtain $\cO'\in \WO(\pi)$.
\end{proof}
For admissible $\pi$ this corollary is \cite[Theorem D]{GGS}.

\DimaK{
In \S \ref{subsec:SL} below we formulate and prove a certain analog of this corollary for $\SL_n(F)$.}


\subsection{Proof of Proposition \ref{prop:GL}}\label{subsec:PfGL}

Let us first introduce some notation.
 A \emph{composition} $\eta$ of $n$ is a sequence of natural (positive) numbers $\eta_1,\dots,\eta_k$ with $\sum\eta_i=n$. The length of $\eta$ is $k$. A  partition $\lambda$ is a composition such that $\lambda_1\geq \lambda_2\geq \dots\geq \lambda_k$.
For a composition $\eta$ we denote by $\eta^\geq$ the corresponding partition.
A partial order on partitions of $n$ is defined by \begin{equation}
\lambda\geq \mu \text{ if }\sum_{i=1}^j \lambda_i\geq \sum_{i=1}^j \mu_i\text{ for any }1\leq j \leq \mathrm{length}(\lambda), \mathrm{length}(\mu).
\end{equation}

We will use the notation $\diag(x_1,\dots,x_k)$ for diagonal and block-diagonal matrices. For a natural number $k$ we denote by $J_k\in \fg_k$ the \emph{lower}-triangular Jordan block of size $k$, and by $h_k$ the diagonal matrix $h_k:=\diag(k-1,k-3,\dots,1-k)$. For a composition $\eta$ we denote \begin{equation}J_{\eta}:=\diag(J_{\eta_1},\dots, J_{\eta_k})\in \fg_n \text{ and }h_{\eta}:= \diag(h_{\eta_1},\dots, h_{\eta_k})\in \fg_n.\end{equation}
Note that $[h_{\eta},J_{\eta}]=-2J_{\eta}$ and $(J_{\eta},h_{\eta})$ can be completed to an $\sl_2$-triple.
Let $E_{ij}$ denote the elementary matrix with 1 in the $(i,j)$ entry and zeros elsewhere.


 Identify $\fg_n^*$ with $\fg_n$ using the trace form. Denote by $\cO_{\eta}$ the orbit of $J_{\eta}$.
 By the Jordan theorem all nilpotent orbits are of this form.
It is well known that $\cO_{\eta}\subset \overline{\cO_{\gam}}$ if and only if $\eta^{\geq} \leq \gam^{\geq}$.
Here, one can take the closure $\overline{\cO_{\gam}}$ in any of the topologies on $\fg_n$ defined by norms on $\Q$, or in the Zariski topology - all these closures coincide.


\begin{lemma}\label{lem:TwoBlocks}
Let $p,q,r \in \Z$ with $p> r\geq0, \, q>0$. Let $Z:=\diag((p+q-r)\Id_p,0_{q+r}))\in \g_{p+q+r}$, $Y:=E_{p+r+1,p}$, $X:=J_{p,q+r}+Y$ and $S:=h_{p,q+r}+Z$. Then $$X,Y\in \fg^S_{-2}, \,\,\, Y\in \fg^Z_{<0}\text{ and }X\in \cO_{p+q,r}.$$
%
\end{lemma}
\DimaN{
\begin{proof}
An operator conjugating $X$ to $J_{p+q,r}$ is given in the standard basis by
\begin{equation}\label{=glconj}
ge_i=\begin{cases}e_i & 1\leq i \leq p \\
e_{i+r} & p<i\leq p+q \\
(-1)^{q+r-1}(e_{i-q-r}-e_{i-q}) & p+q<i\leq p+q+r \\
\end{cases}
\end{equation}
\end{proof}
}
%
%
%
%
%

\begin{lemma}[{\cite[Lemma 4.2.2]{GGS}}]\label{lem:part}
Let $\lambda,\mu$ be partitions of $n$ with $\lambda \geq \mu$. Then there exists an index $i\leq \length(\lambda)$ such that $\lambda_i\geq \mu_i\geq \lambda_{i+1}$. Here, if $i= \length(\lambda)$ we take $\lambda_{i+1}=0$.
\end{lemma}

We are now ready to prove Proposition \ref{prop:GL}.
\begin{proof}[Proof of Proposition \ref{prop:GL}]
We prove the proposition by induction on $n$. The base case $n=1$ is obvious.
For the induction step, assume that the proposition holds for all $n' < n$.
Let $\mu\leq \lambda$ be the partitions corresponding to $\cO'$ and $\cO$.
If $\mu$ has length 1 then $\lambda=\mu$ and the proposition is obvious. If $\mu$ has length 2 then the proposition follows from Lemma \ref{lem:TwoBlocks}. Thus we assume $\length(\mu)\geq 3$. We can also assume that $\lam$ and $\mu$ do not have common parts.

By Lemma \ref{lem:part} there exists an index $i\leq \length(\lambda)$ such that
\begin{equation}\label{=i}
\lambda_i> \mu_i> \lambda_{i+1}.
\end{equation}
Let $n':=n-\mu_i$ and $p:=\lam_i+\lam_{i+1}-\mu_i$.
Let $\mu'$ be the partition of $n'$ obtained from $\mu$ by omitting $\mu_i$ and $\lambda'$ be obtained from $\lambda$ by replacing the two parts $\lambda_i$ and $\lambda_{i+1}$ by a single part $p$. It follows from \eqref{=i} that $\lambda_{i}>p>\lambda_{i+1}$,  and thus $\lam'$ is a partition of $n'$.
\eqref{=i} also implies $\lam'> \mu'$.
Let $\alp$ be the reordering of $\lam'$ obtained by putting the part $p$ on the first place.

Choose a neutral pair $(h',\varphi')$ in $\g_{n'}$ with $\varphi'\in \cO_{\mu'}$. By the induction hypothesis, there exist a rational semi-simple $Z'\in \g_{n'}$ and $\psi'\in  (\fg_{n'})^{Z'}_{<0}\cap (\fg_{n'})^{h'+Z'}_{-2}$ such that $\varphi'+\psi'\in \cO_{\lam'}$. Conjugating by $\GL_{n'}(\Q)$ we may assume that $\varphi'+\psi'=J_{\alp}$. This implies that $Z'$ is diagonal and that the first $p$ entries are equal. By subtracting a scalar matrix, we may assume that  the first $p$ diagonal entries of $Z'$ are zeroes.

Now let $(h,\varphi)$ be a neutral pair with $\varphi\in \cO_{\mu}$. Conjugating by $\GL_{n}(\Q)$ we may assume that
\begin{equation}
\varphi=\diag(J_{\mu_i},\varphi') \text{ and } h=\diag(h_{\mu_i},h').
\end{equation}
Let
\begin{equation}
Z:=\diag((\lam_{i}-\lam_{i+1})\Id_{\mu_i},Z')), \quad Y:=E_{p+\lam_{i+1}+1,p}, \quad  \quad \psi:=Y+\diag(0,\psi') \, \in \fg_n.
\end{equation}

Let us show that $Z$ and $\psi$ satisfy the requirements of the theorem.
Indeed, we have $\psi'\in \fg^{Z'}_{<0}\cap \fg^{h'+Z'}_{-2}$ by construction and $Y\in \fg^{Z}_{<0}\cap \fg^{h+Z}_{-2}$ by Lemma \ref{lem:TwoBlocks}. To see that $\varphi+\psi\in \cO$ note that $\varphi+\psi=Y+J_{\beta}$, where $\beta_1=\mu_i$ and $\beta_{j}=\alp_{j-1}$ for any $j>1$. Decompose $Y+J_{\beta}=\diag(A,J_{\lam''})$, where $A\in \g_{\lam_i+\lam_{i+1}}$, and $\lam''$ is obtained from $\lam'$ by omitting the part $p$.  By   Lemma \ref{lem:TwoBlocks} we have $A\in \cO_{\lam_{i},\lam_{i+1}}$ and therefore $Y+J_{\beta}=\diag(A,J_{\lam''})\in \cO_{\lam}=\cO.$
\end{proof}

\subsection{\DimaK{The case of $\SL_n(F)$}}\label{subsec:SL}
\DimaK{
First of all, let us fix a set of representatives for nilpotent orbits - for an arbitrary field $L$ of characteristic zero, after introducing some notation.

For a composition $\eta=(\eta_1,\dots,\eta_k)$ we denote $d(\eta):=\mathrm{gcd}(\eta_1,\dots,\eta_k)$. For  $a\in L^{\times}$ we denote $D_a:=\diag(a,1,\dots,1)J_{\eta}$ and $J_{\eta}^a:=D_aJ_{\eta}D_a^{-1}$.
Note that $[h_{\eta},J_{\eta}^{a}]=-2J_{\eta}^a$ and $(J_{\eta}^a,h_{\eta})$ can be completed to an $\sl_2$-triple.
Denote 
by $\cO_{\eta}^a$ the $\SL_n(L)$-orbit of $J_{\eta}^a$.

\begin{lem}[cf. {\cite[Proposition 4]{NevSlSp}}]
$\,$
\begin{enumerate}[(i)]
\item Every orbit is of the form $\cO_{\eta}^a$ for some composition $\eta$ and some $a\in L^{\times}$. \item  $\cO_{\eta}^a=\cO_{\alp}^b$ if and only if both $\eta^{\geq}=\alp^{\geq}$ and $a/b\in (L^{\times})^{d(\eta)}$.
\end{enumerate}
\end{lem}

Identify $\sl_n(L)$ with its dual space using the trace form.

\begin{prop}\label{prop:SL}
Let $L$ be any field of characteristic zero, $\fg:=\sl_n(L)$ and let
$\lambda>\mu$ be partitions. Let $d:=\mathrm{gcd}(d(\lambda),d(\mu))$. Let $a,b\in L^{\times}$ such that $a/b\in (L^{\times})^d$. Then for any neutral pair $(h,\varphi)$ with $\varphi\in \cO^b_{\mu} $  there exist a rational semi-simple $Z\in \DimaL{(\g)^h_0\cap \g_{\varphi}}$ and $\psi\in  \fg^Z_{<0}\cap \fg^{h+Z}_{-2}$ such that $\varphi+\psi\in \cO^a_{\lam}$.
\end{prop}
\begin{proof}
Since we can multiply $a$ by $(L^{\times})^{d(\lam)}$ and $b$ by  $(L^{\times})^{d(\mu)}$ without changing the orbits, we can assume $a=b$. Then, applying the automorphism of $\sl_n(L)$ given by conjugation by $\diag(a,1,\dots,1)$, we can assume $a=b=1$. Next, note that, in the notation of Lemma \ref{lem:TwoBlocks},
\DimaN{the matrix $g$ in \eqref{=glconj} that conjugates $X$ to $J_{p+q,r}$ lies in $\SL_n(\Q)$}. Now, the proposition follows by induction in the same way as Proposition \ref{prop:GL}.
\end{proof}

By Theorem \ref{thm:ComparOrbit} we obtain the following corollary.
\begin{corollary}\label{cor:SL}
Let
$\lambda>\mu$ be partitions. Let $d:=\mathrm{gcd}(d(\lambda),d(\mu))$. Let $a,b\in F^{\times}$ such that $a/b\in (F^{\times})^d$.
Let $\pi\in \Rep^{\infty}(\SL_n(F))$ and assume that $\cO_{\lam}^a\in \WS(\pi)$. Then $\cO_{\mu}^b\in \WO(\pi)$.
\end{corollary}

\begin{remark}
In Proposition \ref{prop:SL}, the condition $a/b\in (L^{\times})^d$ is necessary. Indeed, one can show for $n=4, \, \lam=(4), \mu=(2,2), b=1$ and $a\notin (L^{\times})^2$,
no $Z,\psi \in \sl_4(K)$ satisfy the conditions of the proposition.
However, we do not know whether the condition $a/b\in (L^{\times})^d$ is necessary for Corollary \ref{cor:SL}.
\end{remark}
}

\section{Global setting}\label{sec:Glob}
\setcounter{lemma}{0}

\subsection{\DimaO{Basic notions}}

Let $K$ be a number field and let $\A=\A_{K}$ be its ring of adeles. In this section we let $\chi$ be a unitary character  of $\A$, which is trivial on $K$\DimaO{ and such that for  any Archimedean place $\nu$ the restriction $\chi|_{K_\nu}$ of $\chi$ to $K_{\nu}$ is $\exp(2\pi i |x|)$, and for any non-Archimedean place $\nu$, the kernel of $\chi|_{K_\nu}$ is the ring of integers.}
Then $\chi$ defines an isomorphism between $\A$ and $\hat{\A}$ via the map $a\mapsto \chi_{a}$, where $\chi_{a}(b)=\chi(ab)$ for all $b\in \A$. This isomorphism restricts to an  isomorphism
\begin{equation}\label{eq:chi_isomorphism}
 \widehat{\A/K}\cong \{\psi\in \hat{\A}\, | \psi|_{K}\equiv 1\}=\{\chi_{a}\, | \, a\in K\}\cong K.
\end{equation}
Given an algebraic group $\bfG$ defined over $K$ we will denote its Lie algebra by  $\fg$ and we will denote the group of its adelic (resp. $K$-rational) points by $\bfG(\A)$ (resp. $\bfG(K)$). We will also define the Lie algebras $\fg(\A)$ and $\fg(K)$ in a similar way.

Given a Whittaker pair $(S,\varphi)$ on $\g(K)$, we set $\fu=\fg_{\geq 1}^{S}$ and $\fn$ to be the radical of the form $\omega_{\varphi}|_{\fu}$, where $\omega_{\varphi}(X,Y)=\varphi([X,Y])$, as before. Let $\fl\subset \fu$ be any choice of a maximal isotropic Lie algebra with respect to this form, and let $\bfU=\exp \fu$, $\bfN=\exp \fn$ and $\bfL=\exp \fl$. Observe that we can extend $\varphi$ to a linear functional on $\g(\A)$ by linearity and, furthermore, the character $\chi_{\varphi}^{L}(\exp X)=\chi(\varphi(X))$ defined on $\bfL(\A)$ is automorphic, that is, it is trivial on $\bfL(K)$. We will denote its restriction to $\bfN(\A)$ simply by $\chi_{\varphi}$. 

\DimaE{Let $G$ be a finite central extension of $\bfG(\A)$, such that the cover $G\onto \bfG(\A)$ splits over $\bfG(K)$. Fix a discrete subgroup $\Gamma\subset G$ that projects isomorphically onto
$\bfG(K)$. Note that $\bfU(\A)$ has a canonical lifting into $G$, see e.g. \cite[Appendix I]{MW_Scr}.}

\begin{definition}
 Let $(S,\varphi)$ be a Whittaker pair for  $\g(K)$ and let $\bfU,\bfL, \bfN,\chi_{\varphi}$ and $\chi_{\varphi}^{\bf L}$ be as above. For an automorphic function $f$, we define its \emph{$(S,\varphi)$--Fourier coefficient} to be
\begin{equation}\label{eq:Whittaker-Fourier-coefficient}
 \DimaH{\cF}_{S,\varphi}(f):=\int_{N(\A)/N(K)}\chi_{\varphi}(n)^{-1}f(n)dn.
\end{equation}
We also define its \emph{$(S,\varphi,L)$--Fourier coefficient} to be
 \begin{equation}\label{eq:extended_Whittaker-Fourier_coefficient}
 \DimaH{\cF}_{S,\varphi}^{\bf L}(f):=\int_{L(\A)/L(K)}\chi_{\varphi}^{L}(l)^{-1}f(l)dl.
 \end{equation}
Observe that $ \DimaH{\cF}_{S,\varphi}$ and $ \DimaH{\cF}_{S,\varphi}^{\bf L}$ define linear functionals on the space of automorphic forms.
\DimaH{For a subrepresentation $\pi$ of the space of automorphic forms on $G$, we will denote their restrictions to $\pi$ by $ \DimaH{\cF}_{S,\varphi}(\pi)$ and $ \DimaH{\cF}_{S,\varphi}^{\bf L}(\pi)$ respectively.\footnote{To forestall confusion, we emphasize that $L$ here stays for ``Lagrangian" (actually, maximal isotropic), and not Levi.}}
\end{definition}

\DimaO{
\begin{example}
Let $\bfG=\GL_3$. First let $\varphi$ be given by the trace form pairing with the matrix $E_{21}+E_{32}$, where $E_{ij}$ are elementary matrices. Let $h:=2E_{11}-2E_{33}$. Then $(h,\varphi)$ is a neutral Whittaker pair. In this case $N=L=U$ is the group of unipotent upper-triangular matrices and $\cF_{h,\varphi}= \cF^L_{h,\varphi}$ is a classical (non-degenerate) Fourier coefficient.

Now let $\psi$  be given by the trace form pairing with $E_{31}$ and $H=E_{11}-E_{33}$. Then $(H,\psi)$ is another neutral Whittaker pair. In this case $U$ is  the group of unipotent upper-triangular matrices, while $N=\{\Id + cE_{13}\, \vert \, c\in K\}$. There are infinitely many choices for $L$. Two of them are $L_1=\{\Id + bE_{12}+ cE_{13}\, \vert \, b,c\in K\}$ and $L_2=\{\Id +  cE_{13}+d E_{23}\, \vert \, c,d\in K\}$.
\end{example}
}
One defines quasi-Fourier coefficients in a similar way.
In order to adapt our arguments to the global setting we will have to replace Lemma  \ref{lem:Heis} by the following one.

\begin{lemma}[cf. {\cite[Lemma 6.0.2]{GGS}}] \label{lem:Glob}
 \DimaK{Let $\pi$ be subrepresentation  of the space of automorphic forms on $G$.} Let $(S,\varphi,\varphi')$ be a Whittaker triple.
 Then $ \DimaH{\cF}_{S,\varphi,\varphi'}(\pi)\neq 0$ if and only if $ \DimaH{\cF}_{S,\varphi,\varphi'}^{\bf L}(\pi)\neq 0$. More specifically, if $ \DimaH{\cF}_{S,\varphi,\varphi'}(f)\neq 0$ for some $f\in \pi$ then $ \DimaH{\cF}_{S,\varphi,\varphi'}^{\bf L}(\pi(u)f)\neq 0$ for some $u\in U(K)$.
\end{lemma}

Also, note that the global analog of Lemma \ref{lem:BZ} is proven by decomposition to Fourier series. \DimaD{For more details see \cite[\S 5.2]{SH}.}

\DimaK{For two nilpotent $\bfG(K)$-orbits $\cO,\cO'\in \fg(K)$ we will say $\cO'\leq \cO$ if  for any completion $F$ of $K$, the closure of $\cO$ in $\fg(F)$ includes $\cO'$. For a subrepresentation $\pi$ of the space of automorphic forms on $G$, we denote by $\WO(\pi)$ the collection of all nilpotent $\bfG(K)$-orbits in $\fg^*(K)$ such that $\cF_{\varphi}(\pi)\neq 0$ for any $\varphi\in \cO$. We denote the set of maximal orbits in $\WO(\pi)$ by $\WS(\pi)$. }

\subsection{\DimaO{Main results}}
\DimaB{ Repeating the arguments in \S  \ref{sec:MaxOrb}-\ref{sec:cors} we obtain the following theorem.

\begin{thm}\label{thm:Glob}
\DimaH{ Let $\pi$ be subrepresentation  of the space of automorphic forms on $G$, and \DimaC{let  $\varphi\in \fg^*(K)$ be nilpotent.}} Assume that $\bfG(K)\cdot \varphi \in  \WS(\pi)$, and let $f\in \pi$. Then
\begin{enumerate}[(i)]
\item \label{it:GlobDeg} For any rational semi-simple $S\in \fg$ with $ad^*(S)\varphi=-2\varphi$ there exists $f'\in \pi$ such that $ \DimaH{\cF}_{S,\varphi}(f')\neq 0$.
\item \label{it:GlobCusp} If $\pi$ is cuspidal then $\varphi$ does not belong to the Lie algebra of any \DimaM{proper} Levi subgroup of $\bfG(K)$ defined \DimaC{over $K$}.
\end{enumerate}
\end{thm}

For $\bfG=\GL_n$, part \eqref{it:GlobDeg} generalizes \cite[Proposition 5.3]{Cai}. 

Part \eqref{it:GlobCusp} was conjectured in \cite[\S 4]{Ginz}.

\DimaD{
In order to formulate a global analog of Theorem \ref{thm:MaxFin} and deduce an analog of Theorem \ref{thm:adm}, we will introduce the global Weil representation $\varpi$ (\cite{Weil}) and Fourier-Jacobi coefficients, following \cite[\S 5.2]{SH}.

For a symplectic space $V$ over $K$, $\varpi_V$  is the only irreducible unitarizable representation of the double cover Jacobi group $\widetilde{J(V)}:=\widetilde{\Sp(V(\A))}\ltimes \cH(V(\A))$ with central character $\chi$, where $\cH(V)$ is the Heisenberg group of $V$. It has an automorphic realization given by theta functions
$$\theta_{f}(g)=\sum_{a\in \mathcal{E}(K)}\omega_{\chi}(g)f(a), \, \text{where }f\in \Sc(\cE(\A)), \, \mathcal{E} \text{ is a Lagrangian subspace of }V, \text{ and }g\in  \widetilde{J(V)}.$$

Fix an $\sl_2$-triple $\gamma=(x,h,y)$ in $\fg(K)$ and let $\varphi\in \fg^*$ be given by the Killing form pairing with $y$. Let $V:=\fg^h_1$, with the symplectic form $\omega_{\varphi}(A,B):=\varphi([A,B])$.
Then we have a natural map $\ell_{\gamma}:U\rtimes \widetilde{G_{\gamma}}\to \widetilde{J(V)}$.
We define a map $FJ:\pi\otimes \varpi_V\to C^{\infty}\DimaE{(\Gamma \backslash  \widetilde{G_{\gamma}})}$ by
$$f \otimes \eta \mapsto  \int_{U(K)\backslash U(\A)} f(u\tilde g) \theta_{\eta}(\ell_{\gamma}(u,\tilde g))du$$

Then, arguing as in \S \ref{subsec:MaxFin} we obtain from Theorem \ref{thm:Glob} \eqref{it:GlobDeg} the following corollary.
\begin{cor}
If $\bfG(K)\cdot\varphi\in \WS(\pi)$ then the subgroup $\widetilde{M_{\gamma}}$ acts on the image of $FJ$ by $\pm \Id$.
\end{cor}

\DimaN{\begin{cor}
If $\bfG$ is quasi-split over $K$ \DimaR{and semi-simple}, and $f$ is not constant then there exists a neutral Whittaker pair $(h,\varphi)$ with $\varphi\neq 0$ such that $\cF_{h,\varphi}(f)\neq 0$.
\end{cor}}

Since the Weil representation $\varpi_V$ is genuine, the subgroup of $\widetilde{M_{\gamma}}$ that acts  trivially on the image of $FJ$ projects isomorphically on $M_{\gamma}$.
This implies  the following corollary.
\begin{cor}
If $\bfG(K)\cdot\varphi\in \WS(\pi)$ then the cover $\widetilde{M_{\gamma}}$ splits over $M_{\gamma}$.
\end{cor}
Arguing as in the proof of Proposition \ref{prop:ClasQuas} we deduce from this corollary that if \DimaO{$G$ is classical and linear} then all the orbits in $\WF(\pi)$ are special. This was already shown in \cite[Theorem 2.1]{GRS_Sp} and \cite[Theorem 11.2]{JLS}.

\DimaH{
Finally, the following analog of Theorem \ref{thm:ComparOrbit} holds, with an analogous proof.
\DimaQ{
\begin{thm}\label{thm:GlobComparOrbit}
Let $(h,\varphi) \in \fg(K)\times \fg^*(K)$ be a neutral Whittaker pair.
If there exists a Whittaker pair $(S,\Phi)\in \fg(K)\times \fg^*(K)$ such that $\bfG(K)\cdot\Phi\in \WS(\pi)$,  $\varphi\in (\fg^*)^S_{-2}$, $[h,S]=0$, and $\Phi-\varphi\in (\fg^*)^{S-h}_{<0}$ then $\bfG(K) \cdot \varphi \in \WO(\pi).$
\end{thm}
}

As in \S \ref{sec:compar}, this theorem together with Proposition \ref{prop:GL} implies Corollary \ref{intcor:GlobGL}.}
}}
\DimaK{Furthermore, Theorem \ref{thm:GlobComparOrbit} and Proposition \ref{prop:SL} imply the following version of Corollary \ref{intcor:GlobGL} for $SL_n$.

\begin{corollary}\label{cor:SLGlob}
Let
$\lambda>\mu$ be partitions. Let $d:=\mathrm{gcd}(d(\lambda),d(\mu))$. Let $a,b\in K^{\times}$ such that $a/b\in (K^{\times})^d$.
Let $\pi$ be a subrepresentation  of the space of automorphic forms on $\SL_n(\A)$ and assume that $\cO_{\lam}^a\in \WS(\pi)$. Then $\cO_{\mu}^b\in \WO(\pi)$.
\end{corollary}

\DimaU{The methods of this section were further developed in \cite{GSPhysLeviDist,GSPhysEul}.}
}

\subsection{\DimaO{Corollaries for cuspidal representations}}
\DimaO{Theorem \ref{thm:Glob}\eqref{it:GlobCusp} implies the following corollary.}

\begin{cor}\label{cor:GlobCusp}
\DimaH{ Let $\pi$ be a cuspidal subrepresentation  of the space of automorphic forms on $G$, and \DimaC{let  $\varphi\in \fg^*(K)$ s.t.}} $\bfG(K)\cdot \varphi \in  \WS(\pi)$. Then  \DimaM{the quotient of the stabilizer of $\varphi$ in $\bfG(K)$ by the center of $\bfG(K)$} is $K$-anisotropic.  Moreover, assume that $\bfG$ is \DimaL{split and }classical, and let $\lambda$ be the  partition corresponding to $\varphi$. \DimaL{Let $l$ denote the length of $\lam$}. Then
\begin{enumerate}[(i)]
\item \label{it:CuspGL} If $\bfG=\GL_n$ or $\bfG=\SL_n$ \DimaL{then $l=1$}, {\it i.e.} $\pi$ is generic.
\item \label{it:CuspSp} If $\bfG=\Sp_{\DimaN{2n}}$  then $\lambda$ is totally even, \textit{i.e.} consists of even parts only.
\item \label{it:CuspO} If $\bfG=\SO_n$ or $\bfG=\mathrm{O}_n$ then $\lambda$ is totally odd, \DimaL{and the multiplicity of each part does not exceed $(l+1)/2$}.
\end{enumerate}
\end{cor}
\DimaL{
\begin{proof}
First assume by way of contradiction that \DimaM{the stabilizer of $\varphi$ in $\bfG(K)$ includes a $K$-split torus $T$ that is not central in $\bfG(K)$. Let $L$ be the centralizer of $T$\ in $\bfG(K)$. Then $L$ is a proper Levi subgroup and its Lie algebra includes $\varphi$, contradicting Theorem \ref{thm:Glob}\eqref{it:GlobCusp}. Thus, the quotient of the stabilizer of $\varphi$ in $\bfG(K)$ by the center of $\bfG(K)$ is $K$-anisotropic.}

This immediately implies \eqref{it:CuspGL}.

For \eqref{it:CuspSp} and \eqref{it:CuspO} let $V$ denote the standard representation of $\bfG(K)$, $\omega$ denote the bilinear form on $V$ and $\eps$ the sign of $\omega$, {\it i.e.} $\omega(v,w)=\eps\omega(w,v)$. Choose an $\sl_2$-triple $\gamma=(e,h,f)\in \fg(K)$ such that $\varphi$ is given by the Killing form pairing with $f$. Then $\gamma$ defines a decomposition $$V=\bigoplus_{i=1}^kV_i\otimes W_i,$$
where $V_i$ is the irreducible $i$-dimensional representation of $\sl_2(K)$ and $W_i$ is the multiplicity space of $V_i$ in $V$. Note that $m_i:=\dim W_i$ is the multiplicity of $i$ in $\lambda$ (which might be 0). The form $\omega$ on $V$ defines a bilinear form $\omega_i$ on each $W_i$, that satisfies $\omega_i(v,w)=\eps(-1)^i\omega_i(w,v)$. Since the \DimaM{center of $\bfG(K)$ is finite,} the stabilizer of $\varphi$ in $\bfG(K)$ is $K$-anisotropic, \DimaM{thus so is the centralizer of $\gamma$, and thus all the}  $(W_i,\omega_i)$ are anisotropic. In particular, none of the forms $\omega_i$ is symplectic. This implies that $\lambda$ is totally even in case \eqref{it:CuspSp} and totaly odd in case \eqref{it:CuspO}.

To finish the proof, assume by way of contradiction that $\eps=1$, $\lambda$ is totally odd and $m_i>(l+1)/2$ for some $i$. Let $(H,Q)$ denote the 2-dimensional quadratic space with $Q(x,y):=xy$ and $(K,q)$ denote the one-dimensional quadratic space with $q(x)=x^2$. Let $(W,\omega'):=\bigoplus (W_i,\omega_i)$.
Then $\dim W=l$. Let $n=2k+j,$ where $j\in \{0,1\}$.   Then
$$H^k\oplus K^j\simeq (V,\omega)\simeq (W,\omega')\oplus H^{(n-l)/2}.$$  By Witt's cancelation theorem, this implies $H^{(l-j)/2}\oplus K^j\simeq (W,\omega')$. Let $U:=W_i\cap H^{(l-j)/2}$. Then $\dim U\geq m_i+(l-j)-l=m_i-j>(l+1)/2-j\geq (l-j)/2$, and thus $U$ includes an isotropic vector.  This contradicts the condition that $W_i$ is anisotropic.
\end{proof}
For $\GL_n$ this is a classical result of Piatetski-Shapiro, and the case of $\Sp_{\DimaO{2n}}$ was shown  in \cite{GRS_Sp,Shen}.

\DimaU{
\begin{remark}
Corollary \ref{cor:GlobCusp} implies that the smallest possible partition in the Whittaker support of a cuspidal automorphic representation $\pi$ of $\Sp_{2n}$ is $2^n$. This bound is sharp for even $n$ by \cite{Ike}. For $SO_{n,n}$ and $O_{n,n}$ we obtain the lower bound $3^{n/2}1^{n/2}$ if $n$ is even and $53^{(n-3)/2}1^{(n-1)/2} $ if $n$ is odd. For $O_{n+1,n}$ and $SO_{n+1,n}$ obtain the lower bound $3^{n/2}1^{n/2+1}$ if $n$ is even and $3^{(n+1)/2}1^{(n-1)/2} $ if $n$ is odd. These
bounds are conjectured to be sharp in \cite[Conjecture 2.14]{JL}.
\end{remark}
}
\DimaN{
\begin{cor}
Let $\bf G$ be a split classical group of rank at least 3. Let $\pi$ be an irreducible automorphic representation of $G$, and let $\pi=\bigotimes_{\nu}\pi_{\nu}$ be its decomposition to local factors. Suppose that for some place $\nu$, every orbit in $\WS(\pi_{\nu})$ lies inside the Zariski closure of a complex next-to-minimal orbit.
Then $\pi$ cannot be realized in the cuspidal spectrum.
\end{cor}
\begin{proof}
It is easy to see that any orbit in $\WO(\pi)$ lies inside some orbit in $\WO(\pi_{\nu})$. Indeed, for any neutral Whittaker pair $(h,\varphi)$ and any maximal isotropic subspace $\fl\subset \fg$, the Fourier coefficient $\cF_{h,\varphi}^L$ is an $L(K)$-equivariant functional on $\pi$. Its existence implies the existence of an $L(K_{\nu})$-equivariant functional on $\pi_{\nu}$.

Thus, any orbit in $\WS(\pi)$ has a minimal or a next-to-minimal partition. For $\GL_n$, $\SL_n$ and $\Sp_n$, with $n\geq 4$ the minimal partition is $(2,1^{n-2})$ and the next-to-minimal one is $(2^2,1^{n-4})$. Both are different from $(n)$, and include odd parts if $n\geq 6$. For $SO_{k,k+i}$ with $i\in \{0,1\}$ and $k\geq 2$, the minimal partition is $(2^2,1^{2k+i-4})$ and thus includes even parts. There are at most two  next-to-minimal partitions for $SO_{k,k+i}$. One is  $(3,1^{2k+i-3})$ which has length $l:=n-2$, and for $n\geq 6$ we have $m:=n-3>(l+1)/2$. The other is $2^4 1^{2k+i-8}$ (it only appears for $k\geq 4$) and it includes even parts. The Corollary follows now from Corollary \ref{cor:GlobCusp}.
\end{proof}

\DimaU{Most of this corollary can also be deduced from \cite{JSL}.}
In \cite[Corollary G]{GSPhys} an analogous statement is proven for $E_6,E_7$ and $E_8$, by expressing any minimal or next-to-minimal automorphic form on these groups through its  Whittaker-Fourier coefficients, \emph{i.e.} period integrals over the nilradical of a Borel subgroup of $G$ against a character of this subgroup, following \cite{MilSah,AGKLP}. We expect this corollary to extend to $F_4$ and $G_2$ as well.
Our interest in minimal and next-to-minimal representations is driven by the special role played by them and by their Fourier coefficients in string theory, cf.  \cite{GMV} and \cite[Part II]{FGKP}.
}
}


\end{document}